\documentclass[a4paper,10pt]{amsart}
\usepackage[utf8]{inputenc}
\usepackage{verbatim}
\usepackage{amssymb,amsmath}
\usepackage{enumerate}
\usepackage[justification=centering]{caption}
\usepackage{cite}

\usepackage[active]{srcltx}
\DeclareMathOperator{\cf}{cf}

\DeclareMathOperator{\chr}{\chr}

\DeclareMathOperator{\Ht}{ht}
\DeclareMathOperator{\lex}{lex}

\newcommand{\llex}{<_{\lex}}

\numberwithin{equation}{section}
\usepackage{float}
\usepackage[left=3.5 cm, right=3.5cm]{geometry}
 \usepackage[usenames,dvipsnames]{pstricks}
 \usepackage{epsfig}
\usepackage{pst-grad}
\usepackage{pstricks,pst-node}

\usepackage{xmpmulti}
\usepackage{psfrag}
\usepackage{tikz}

\usepackage{mathtools}

\newcommand{\tradd}[2]{#1\oplus #2}

\makeatletter
\def\slashedarrowfill@#1#2#3#4#5{%
  $\m@th\thickmuskip0mu\medmuskip\thickmuskip\thinmuskip\thickmuskip
   \relax#5#1\mkern-7mu%
   \cleaders\hbox{$#5\mkern-2mu#2\mkern-2mu$}\hfill
   \mathclap{#3}\mathclap{#2}%
   \cleaders\hbox{$#5\mkern-2mu#2\mkern-2mu$}\hfill
   \mkern-7mu#4$%
}

\def\rightslashedarrowfilla@{%
  \slashedarrowfill@\relbar\relbar{\raisebox{1.2pt}{$\scriptscriptstyle\diagup$}}\rightarrow}
\newcommand\xslashedrightarrowa[2][]{%
  \ext@arrow 0055{\rightslashedarrowfilla@}{#1}{#2}}

\def\rightslashedarrowfillb@{%
  \slashedarrowfill@\relbar\relbar/\rightarrow}
\newcommand\xslashedrightarrowb[2][]{%
  \ext@arrow 0055{\rightslashedarrowfillb@}{#1}{#2}}

\def\rightslashedarrowfillc@{%
  \slashedarrowfill@\relbar\relbar{\raisebox{.12em}{\tiny/}}\rightarrow}
\newcommand\xslashedrightarrowc[2][]{%
  \ext@arrow 0055{\rightslashedarrowfillc@}{#1}{#2}}

\pgfdeclareshape{slash underlined}
{
  \inheritsavedanchors[from=rectangle] 
  \inheritanchorborder[from=rectangle]
  \inheritanchor[from=rectangle]{north}
  \inheritanchor[from=rectangle]{north west}
  \inheritanchor[from=rectangle]{north east}
  \inheritanchor[from=rectangle]{center}
  \inheritanchor[from=rectangle]{west}
  \inheritanchor[from=rectangle]{east}
  \inheritanchor[from=rectangle]{mid}
  \inheritanchor[from=rectangle]{mid west}
  \inheritanchor[from=rectangle]{mid east}
  \inheritanchor[from=rectangle]{base}
  \inheritanchor[from=rectangle]{base west}
  \inheritanchor[from=rectangle]{base east}
  \inheritanchor[from=rectangle]{south}
  \inheritanchor[from=rectangle]{south west}
  \inheritanchor[from=rectangle]{south east}
  \inheritanchorborder[from=rectangle]
  \foregroundpath{
    \southwest \pgf@xa=\pgf@x \pgf@ya=\pgf@y
    \northeast \pgf@xb=\pgf@x \pgf@yb=\pgf@y
    \pgf@xc=\pgf@xa
    \advance\pgf@xc by .5\pgf@xb
    \pgf@yc=\pgf@ya
    \advance\pgf@xc by -1.3pt
    \advance\pgf@yc by -1.8pt
    \pgfpathmoveto{\pgfqpoint{\pgf@xc}{\pgf@yc}}
    \advance\pgf@xc by  2.6pt
    \advance\pgf@yc by  3.6pt
    \pgfpathlineto{\pgfqpoint{\pgf@xc}{\pgf@yc}}
    \pgfpathmoveto{\pgfqpoint{\pgf@xa}{\pgf@ya}}
    \pgfpathlineto{\pgfqpoint{\pgf@xb}{\pgf@ya}}
 }
}
\tikzset{nomorepostaction/.code=\let\tikz@postactions\pgfutil@empty}

\makeatother

\newcommand{\vareps}{\varepsilon}

\newtheorem{theorem}{Theorem}
\newtheorem{prop}[theorem]{Proposition}

\newtheorem{cor}[theorem]{Corollary}
\newtheorem{con}[theorem]{Conjecture}
\newtheorem{claim}[theorem]{Claim}
\newtheorem{obs}[theorem]{Observation}
\newtheorem{mlemma}[theorem]{Main Lemma}

\newtheorem{tlemma}[theorem]{Lemma}
\newtheorem{tclaim}[theorem]{Claim}

\newtheorem{prob}[theorem]{Problem}
\newcommand{\mf}[1]{\mathfrak{#1}}

\newcommand{\mc}[1]{\mathcal{#1}}
\newcommand{\mb}[1]{\mathbb{#1}}

\newcommand{\oo}{\omega}
\newcommand{\uhr}{\upharpoonright}
\newcommand{\omg}{{\omega_1}}
\newcommand{\dom}{\text{dom}}
\newcommand{\ran}{\text{ran}}
\newcommand{\onto}{\twoheadrightarrow}

\newcommand{\BA}{\textmd{BA}}

\newcommand{\setm}{\setminus}

\newcommand{\tri}{\vartriangleleft}

 \newtheorem{dfn}[theorem]{Definition}
\begin{document}

\title{Uncountable strongly surjective linear orders}

  \author{D\'aniel T. Soukup}
  \address[D.T. Soukup]{Universität Wien,
 Kurt Gödel Research Center for Mathematical Logic, Austria}
  \email[Corresponding author]{daniel.soukup@univie.ac.at}
 \urladdr{http://www.logic.univie.ac.at/$\sim  $soukupd73/}




\begin{abstract}
We call a linear order $L$ \emph{strongly surjective} if whenever $K$ is a suborder of $L$ then there is a surjective $f:L\to K$ so that $x\leq y$ implies $f(x)\leq f(y)$.  We prove various results on the existence and non-existence of uncountable strongly surjective linear orders answering questions of R. Camerlo, R. Carroy and A. Marcone. In particular, $\diamondsuit^+$ implies the existence of a lexicographically ordered Suslin tree which is strongly surjective and minimal; every strongly surjective linear order must be an Aronszajn type under $2^{\aleph_0}<2^{\aleph_1}$ or in the Cohen and other canonical models (where $2^{\aleph_0}=2^{\aleph_1}$); finally, we prove that it is consistent with CH that there are no uncountable strongly surjective linear orders at all. 
\end{abstract}
\maketitle

\section{Introduction}

The study of \emph{minimal uncountable linear orders} goes back several decades. Recall that an uncountable linear order $L$ is minimal if $L$ embeds into any of its suborders $K$ i.e. there is an order preserving injection from $L$ into $K$. Minimal linear orders have been studied extensively and beautiful techniques emerged from these investigations (see \cite{baum, ARSh, 5basis, onlymin} for example).  

The dual notion was only recently introduced by R. Camerlo, R. Carroy and A. Marcone \cite{raphael0}: a linear order $L$ is \emph{strongly surjective} if for all $K\subseteq L$, there is surjective $f:L\to K$ so that $x\leq y$ implies $f(x)\leq f(y)$ (we will use the notation $L\onto K$ to denote this). The reader is invited to show that, for example, $\omega$ and $\mb Q$ are strongly surjective but $\omega+1$ is not. Strongly surjective ordinals were characterized in \cite[Corollary 29]{raphael0} and then the authors proceeded by studying strong surjectivity in greater generality \cite{raphael}. In fact, it was shown that there is a great variety of countable strongly surjective linear orders (beside ordinals).

The main question we tackle in our paper is as simple as this: are there any  strongly surjective linear orders which are not countable? If so, what are the possible order types of these linear orders? It was observed in \cite{raphael} already that any strongly surjective linear order $L$ must be \emph{short} i.e. contains no uncountable well ordered subset; this implies that $|L|\leq \mf c=2^{\aleph_0}$ (see \cite[Theorem 3.4]{stevotrees}). Now, an uncountable short linear order either contains an uncountable real order type, or it is a so called \emph{Aronszajn type}\footnote{\emph{Specker types} are also used as an alternative name for the same notion.} (by definition). In turn, we will focus on these two types.

Our first goal is to look at real order types with regards to strong surjectivity. It was proved in \cite{raphael} that if the Proper Forcing Axiom holds then any $\aleph_1$-dense set $X\subset \mb R$ is strongly surjective; furthermore, $X$ is \emph{minimal} under PFA as well \cite{baumiso}. Now, how does minimality and strong surjectivity relate to one another? First in Section \ref{sec:real}, we present, under certain set-theoretic assumptions, an uncountable strongly surjective $X\subset \mb R$ which is not minimal. However, at this point, we are not aware of a homogeneous and minimal linear order which is not strongly surjective.

Now, the Continuum Hypothesis (i.e. $\mf c=\aleph_1$) implies that every strongly surjective $X\subset \mb R$  is actually countable \cite{raphael}. In particular, uncountable strongly surjective real orders may or may not exist. So, what is the exact role of CH here? Firstly, we prove that $\mf c<2^{\aleph_1}$ already implies a stronger result: every uncountable strongly surjective linear order must be Aronszajn. Moreover, we prove that no strongly surjective linear order contains a real suborder of size $\mf c$ in ZFC. Both statements will follow from our Lemma \ref{lm:ineq}. 
Second, we prove that $\mf c=\aleph_2$ is also consistent with the statement that ``every uncountable strongly surjective linear order is Aronszajn"; this can be achieved using Cohen-forcing or other canonical models (e.g. countable support iteration of Sacks-or Miller-forcing) where we apply the technique of parametrized weak diamonds. These topics are discussed in Section \ref{sec:allAronszajn}.

Next, looking at Aronszajn types our main result is the following: the guessing principle $\diamondsuit^+$ implies that there is a ccc\footnote{I.e. there is no uncountable family of pairwise disjoint nonempty intervals.}, strongly surjective Aronszajn linear order $L$.  In \cite{raphael}, the authors outlined a plan to achieve this result building on a theorem of J. Baumgartner, who constructed a minimal Aronszajn type from $\diamondsuit^+$ (in the form of a lexicographically ordered Suslin tree). However, Baugartner's original proof  has a serious gap as pointed out by Hossein Lamei Ramandi (in fact, Baumgartner's crucial \cite[Lemma 4.14]{baum} is false). We intend to present a complete proof for both of these results in Section \ref{sec:diamond}. 

At this point, we achieved that there could be uncountable real or Aronszajn linear orders which are strongly surjective, but CH only allows the second type. Is it possible that there is a model of CH where there are no uncountable strongly surjective linear orders at all? Our final main result is a positive answer to this question: we show, in Section \ref{sec:nouncountble}, that CH together with J. Moore's axiom (A) \cite{onlymin} implies that every strongly surjective linear order is countable.

Our paper closes with a healthy list of open problems.

\subsection{Acknowledgements}

We thank S. Friedman, M. Hru\v s\'ak, J. Moore and A. Rinot for helpful discussions and remarks. We are especially grateful to R. Carroy for stimulating conversations throughout this project. Finally, we are grateful for the anonymous referee's careful reading and many useful comments.

The author was supported in part by the FWF Grant I1921 and OTKA grant no.113047.

  \subsection{Notations}
  
  We use standard set theoretic notation and, in general, we refer the reader to the classic \cite{jech} for undefined notions. For a map $f:X\to Y$ and $y\in Y$ we let $f^{-1}(y)=\{x\in X:f(x)=y\}$.
  
  If not stated otherwise, all linear orders considered in this paper are infinite. Given a linear order $L$ with order $<$, we use the notation $[x,y]_<$ to denote the closed interval between $x$ and $y$ i.e. $\{z\in L:x\leq z\leq y\}$. We say that $D$ is dense in $L$ iff $x<y$ implies that $x<d<y$ for some $d\in D$; $D$ is \emph{$\kappa$-dense} if  $x<y$ implies that $x<d<y$ for $\kappa$-many $d\in D$.
  
  Given a linear order $X$ and linear orders $K_x$ for each $x\in X$ we can define $L=\sum_{x\in X} K_x$ as the disjoint union $\sqcup_{x\in X} K_x$ so that $a<_L b$ iff $a\in K_x,b\in K_y$ and $x<_X y$ or $x=y$ and $a<_{K_x} b$. Note that the map $f:\sum_{x\in X} K_x\to X$ where $f\uhr K_x=x$ witnesses $\sum_{x\in X} K_x\onto X$. On the other hand, if $f:L\onto X$ then $L=\sum_{x\in X}f^{-1}(x)$ and any transversal $\ell_x\in f^{-1}(x)$ defines a copy $\{\ell_x:x\in X\}$ of $X$ inside $L$.

  In Section \ref{sec:nouncountble}, we work with Aronszajn trees  i.e. $\aleph_1$-trees without uncountable branches. We say that a tree is \emph{Hausdorff} if there are no branching at limit levels; in contrast to the literature, we often work with non Hausdorff trees. We refer the reader for more on trees and their relation to linear orders to \cite{stevotrees}.

\section{On strongly surjective real suborders}\label{sec:real}

We mentioned already in the introduction that it is independent of ZFC whether there is an uncountable, strongly surjective suborder of $\mb R$ \cite{raphael}. Furthermore, strongly surjective and minimal linear orders seem to go hand-in-hand as far as suborders of $\mb R$ are concerned.
Our first goal is to further clarify when uncountable strongly surjective suborders of $\mb R$ exist and if so, how are they related to minimal orders. 

Let us recall a theorem from \cite{raphael} first.


\begin{theorem}[Proposition 5.14 \cite{raphael}]\label{raphreal}
    Suppose that $X\subseteq \mb R$ is the unique $\kappa$-dense suborder of the reals up to isomorphism. Then $X$ is strongly surjective (and minimal if $\kappa=\aleph_1$).
   \end{theorem}

   The existence of a unique $\aleph_1$-dense suborder of $\mb R$ was proved consistent by J. Baumgartner \cite{baumiso} in 1973. Let $\BA_\kappa$ denote the statement that any two $\kappa$-dense subsets of $\mb R$ are isomorphic. $\BA_{\aleph_1}$ is a consequence of PFA (but requires no large cardinals) and $\BA_{\mf c}$ fails in ZFC. The consistency of $\BA_{\aleph_2}$ was only recently announced by I. Neeman and his proof uses large cardinals. However, the consistency of $\BA_\kappa$ is still open for $\kappa>\aleph_2$; in turn, Theorem  \ref{raphreal} at the moment  only implies the existence of strongly surjective, real suborders of size $\leq \aleph_2$. 

Now, under $\BA_{\aleph_1}$, any $\aleph_1$-dense suborder $L$ of $\mb R$ is strongly surjective and minimal. Our first goal is to show the following.

\begin{theorem}\label{thm:notmin}
 Consistently, there is an  $\aleph_1$-dense strongly surjective suborder $L$ of $\mb R$ which is not minimal.
\end{theorem}

We start by a definition: a set $A\subseteq \mb R$ is \emph{increasing} if whenever $n\in \oo$ and $\{a_\xi:\xi<\omg\}\subseteq [A]^{n}$ then there is $\alpha<\beta$ so that $a_\alpha(i)\leq a_\beta(i)$ for all $i<n$. Here, $\{a_\alpha(i):i<n\}$ is the increasing enumeration of $a_\alpha$.

   \begin{tlemma}[\cite{abMA}]\label{isa}
    If $A$ is increasing then $A$ and $-A=\{-x:x\in A\}$ has no common uncountable suborder. 
   \end{tlemma}

   In particular, $BA_{\aleph_1}$ fails if there is an uncountable, increasing $A\subseteq \mb R$.
   
   \begin{theorem}[Theorem 3.1 and 6.2 \cite{ARSh}]\label{ARSh} Consistently, $MA_{\aleph_1}+OCA$ holds and there is an $\aleph_1$-dense, increasing $A\subseteq \mb R$. Furthermore, if $B$ is $\aleph_1$-dense and does not contain a copy of $-A$ then $B$ is isomorphic to $A$.
   \end{theorem}

   Suppose that $A$ is as in Theorem \ref{ARSh}; note that Lemma \ref{isa} implies that $A$ is minimal. 
   
   \begin{prop}
    $A$ is strongly surjective.
   \end{prop}
   \begin{proof}
   Since $A$ is dense, $A\onto \mb Q$ and so $A\onto X$ for any countable linear order $X$ as well.
   
   Now, let $X\subseteq A$ be uncountable. Then there is a countable $Y\subseteq X$ so that $X\setm Y$ is $\aleph_1$-dense. Since $X\setm Y$ cannot contain a copy of $-A$, $X\setm Y$ must be isomorphic to $A$.  Now let $K=\sum_{x\in X}K_x$ where 
   \[
     K_x= \begin{cases}
 		     \{x\}, & \text{for } x\in X\setm Y, \text{ and} \\
 		    A,  & \text{for } x\in Y.
         \end{cases}
   \] $K$ is still a suborder of $\mb R$ and contains no copies of $-A$; indeed, $K$ is a countable union of copies of $A$. So $K$ is again isomorphic to $A$. However, $K\onto X$ and so $A\onto X$ as desired.

   \end{proof}
   
   Hence, $\BA_{\aleph_1}$ can fail while there are uncountable, strongly surjective suborders of $\mb R$. However, $A$ was minimal so this is not a big surprise.


\begin{proof}[of Theorem \ref{thm:notmin}]
Suppose that $A$ is the increasing set from the model of Theorem \ref{ARSh}. It suffices to find a strongly surjective $ L\subseteq \mb R$ which contains both $A$ and $-A$. Then $L$ is clearly not minimal.

   \begin{tlemma}\label{lm:glue}
 Suppose that $L_i$ is a strongly surjective linear order without endpoints for $i<\oo$. Then there is a single strongly surjective $L$ so that $L$ contains a copy  $L_i$ for any $i<\oo$.
\end{tlemma}
\begin{proof}
 We define $L=\sum_{q\in \mb Q} K_q$ so that each $K_q$ is a copy of some $L_i$ and $\{q\in \mb Q:K_q$ is a copy of $L_i\}$ is dense in $\mb Q$ for all $i<\oo$. In particular, $\sum_{q\in I} K_q$ is isomorphic to $L$ for any open interval $I\subseteq \mb Q$.
 
 Now fix any $X\subseteq L$. First, let $A=\{r:X\cap K_r\neq \emptyset\}$; note that $X=\sum_{r\in A}{X\cap K_r}$. Take $h:\mb Q\onto A$ so that $h^{-1}(q)$ is an open interval in $\mb Q$ for all $q\in A$ (this can be easily done using that $\mb Q=\sum_{q\in\mb Q}\mb Q$ and that $\mb Q$ is strongly surjective). 
 
 We will  find $f:L\onto X$ so that if $x\in K_q$ and $r=h(q)$ then $f(x)\in K_r\cap X$; see Figure \ref{fig:lmglue} (A).

 \begin{figure}[H]
  \psscalebox{0.7} 
{
\begin{pspicture}(0,-2.385)(19.8,2.385)
\rput[bl](9.28,0.195){\Large{$(B)$}}
\psellipse[linecolor=black, linewidth=0.04, dimen=outer](12.08,1.445)(0.48,0.19)
\psellipse[linecolor=black, linewidth=0.04, dimen=outer](14.28,1.445)(0.48,0.19)
\psellipse[linecolor=black, linewidth=0.04, dimen=outer](16.48,1.445)(0.48,0.19)
\rput[bl](18.66,-0.165){\Large{$L_i$}}
\rput[bl](17.8,1.875){$\sum_{h(q)=r}\widetilde K_q$}
\rput[bl](14.06,1.935){$\tilde K_{q_0}$}
\rput[bl](16.26,1.935){$\tilde K_{q_1}$}
\rput[bl](11.86,1.935){$\tilde K_{q_{-1}}$}
\psdots[linecolor=black, dotsize=0.2](13.26,-0.465)
\psdots[linecolor=black, dotsize=0.2](15.06,-0.465)
\psdots[linecolor=black, dotsize=0.2](16.86,-0.465)
\psdots[linecolor=black, dotsize=0.2](11.46,-0.465)
\rput[bl](13.06,-1.265){$z_0$}
\rput[bl](11.26,-1.265){$z_{-1}$}
\rput[bl](14.86,-1.265){$z_1$}
\rput[bl](16.86,-1.265){$z_2$}
\psline[linecolor=black, linewidth=0.03, tbarsize=0.07055555cm 5.0,rbracketlength=0.15]{(-)}(12.64,1.435)(13.64,1.435)
\psline[linecolor=black, linewidth=0.03, tbarsize=0.07055555cm 5.0,rbracketlength=0.15]{(-)}(14.88,1.435)(15.88,1.435)
\psdots[linecolor=black, dotsize=0.1](11.2,1.495)
\psdots[linecolor=black, dotsize=0.1](10.76,1.495)
\psdots[linecolor=black, dotsize=0.1](10.32,1.495)
\psdots[linecolor=black, dotsize=0.1](18.14,1.415)
\psdots[linecolor=black, dotsize=0.1](17.7,1.415)
\psdots[linecolor=black, dotsize=0.1](17.26,1.415)
\psline[linecolor=black, linewidth=0.03, tbarsize=0.07055555cm 5.0,rbracketlength=0.15]{(-)}(11.78,-0.465)(13.02,-0.465)
\psline[linecolor=black, linewidth=0.03, tbarsize=0.07055555cm 5.0,rbracketlength=0.15]{(-)}(13.54,-0.465)(14.78,-0.465)
\psline[linecolor=black, linewidth=0.03, tbarsize=0.07055555cm 5.0,rbracketlength=0.15]{(-)}(15.36,-0.465)(16.6,-0.465)
\psdots[linecolor=black, dotsize=0.1](18.36,-0.465)
\psdots[linecolor=black, dotsize=0.1](17.92,-0.465)
\psdots[linecolor=black, dotsize=0.1](17.48,-0.465)
\psdots[linecolor=black, dotsize=0.1](10.8,-0.465)
\psdots[linecolor=black, dotsize=0.1](10.36,-0.465)
\psdots[linecolor=black, dotsize=0.1](9.92,-0.465)
\psline[linecolor=black, linewidth=0.03, linestyle=dashed, dash=0.17638889cm 0.10583334cm, arrowsize=0.05291667cm 2.0,arrowlength=1.4,arrowinset=0.0]{->}(12.12,1.115)(12.32,-0.245)
\psline[linecolor=black, linewidth=0.03, linestyle=dashed, dash=0.17638889cm 0.10583334cm, arrowsize=0.05291667cm 2.0,arrowlength=1.4,arrowinset=0.0]{->}(13.1,0.995)(13.26,0.015)
\psline[linecolor=black, linewidth=0.03, linestyle=dashed, dash=0.17638889cm 0.10583334cm, arrowsize=0.05291667cm 2.0,arrowlength=1.4,arrowinset=0.0]{->}(14.22,1.095)(14.14,-0.325)
\psline[linecolor=black, linewidth=0.03, linestyle=dashed, dash=0.17638889cm 0.10583334cm, arrowsize=0.05291667cm 2.0,arrowlength=1.4,arrowinset=0.0]{->}(15.32,0.975)(15.18,0.035)
\psline[linecolor=black, linewidth=0.03, linestyle=dashed, dash=0.17638889cm 0.10583334cm, arrowsize=0.05291667cm 2.0,arrowlength=1.4,arrowinset=0.0]{->}(16.42,1.055)(16.08,-0.285)
\psline[linecolor=black, linewidth=0.03, linestyle=dashed, dash=0.17638889cm 0.10583334cm, arrowsize=0.05291667cm 2.0,arrowlength=1.4,arrowinset=0.0]{->}(14.14,-1.045)(14.14,-2.025)
\psline[linecolor=black, linewidth=0.04](11.88,-2.265)(16.68,-2.265)
\rput[bl](17.2,-2.385){$\emptyset\neq X\cap K_r\subseteq L_i$}
\psline[linecolor=black, linewidth=0.04](1.2,0.775)(7.2,0.775)
\psline[linecolor=black, linewidth=0.04](1.8,-1.025)(6.6,-1.025)
\rput[bl](0.0,-0.425){\Large{$(A)$}}
\rput[bl](7.6,1.175){\Large{$L$}}
\psellipse[linecolor=black, linewidth=0.032, linestyle=dashed, dash=0.17638889cm 0.10583334cm, dimen=outer](4.1,0.775)(1.7,0.4)
\rput[bl](7.6,-0.825){\Large{$X$}}
\psellipse[linecolor=black, linewidth=0.03, linestyle=dashed, dash=0.17638889cm 0.10583334cm, dimen=outer](4.5,-0.975)(0.9,0.45)
\rput[bl](4.14,-2.325){$\emptyset\neq X\cap K_r\subseteq L_i$}
\rput[bl](0.8,1.375){$\sum_{h(q)=r}K_q$}
\psdots[linecolor=black, dotsize=0.2](4.4,0.775)
\psdots[linecolor=black, dotsize=0.2](4.64,-1.005)
\psline[linecolor=black, linewidth=0.04, arrowsize=0.05291667cm 2.0,arrowlength=1.4,arrowinset=0.0]{->}(4.4,0.575)(4.64,-0.885)
\end{pspicture}
}\caption{The proof of Lemma \ref{lm:glue}}
\label{fig:lmglue}
 \end{figure}
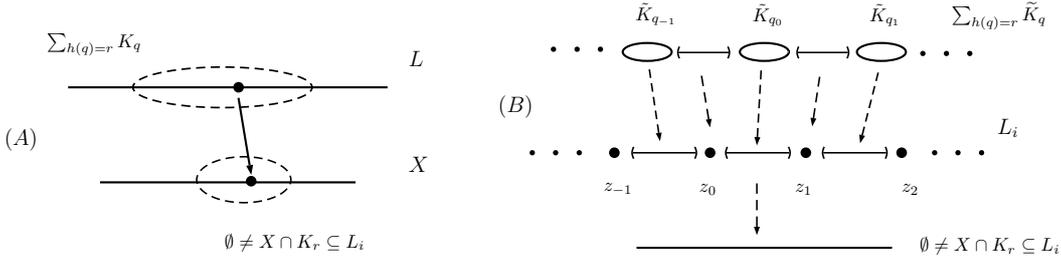

 So fix $r\in A$ and suppose that $X\cap K_r\subseteq L_i$. Our goal now is to define a surjection $f_r:\sum_{h(q)=r}K_q\onto X\cap K_r$; in the end, we will take $f=\bigcup_{r\in A}f_r$.
 
 First, we let $\tilde K_q$ be a copy of $L_i$ if  $K_q$ was isomorphic to $L_i$ and $\tilde K_q$ singleton otherwise for any $q$ with $h(q)=r$. Now, $$\sum_{h(q)=r}K_q\onto \sum_{h(q)=r} \tilde K_q$$ clearly holds. So, it suffices to show that there is some $g:\sum_{h(q)=r} \tilde K_q\onto L_i$ since $L_i\onto  X\cap K_r$ holds as $L_i$ is strongly surjective.
 
So let us pick $(q_k)_{k\in \mb Z}\subseteq h^{-1}(r)$ cofinal, coinitial of type $\mb Z$ so that $\tilde K_{q_k}$ is a copy of $L_i$ and also a cofinal, coinitial $(z_k)_{k\in \mb Z}$ in $L_i$. Define $g$ so that $g\uhr \sum_{q\in (q_{k-1},q_k)_{\mb Q}}\tilde K_q$ is constant $z_k$ and $g\uhr \tilde K_{q_k}:\tilde K_{q_k}\onto (z_k,z_{k+1})_{L_i}$ (the latter exists since $L_i$ is strongly surjective); see Figure \ref{fig:lmglue} (B). 

This ends the construction of $f_r:\sum_{h(q)=r}K_q\onto X\cap K_r$ and hence the construction of $f:L\onto X$.
\end{proof}
This finishes the proof of Theorem \ref{thm:notmin}.

\end{proof}

At this point, it is unclear if a minimal (homogeneous) uncountable $L\subseteq \mb R$ is necessarily strongly surjective as well.

\smallskip

A well-studied strengthening of being increasing is the following:  we say that $A$ is \emph{$k$-entangled} if for any uncountable set $\{(a^{\xi}_0\dots a^{\xi}_{k-1}):\xi<\omg\}$ of $k$-tuples from $A$ and any $\vareps:k\to 2$ there are $\xi<\zeta<\omg$ so that $$a^\xi_i<a^\zeta_i \textmd{ iff } \vareps(i)=0$$ for all $i<k$ (see \cite{ARSh}).  Note that if $L$ is strongly surjective then $L$ is short and being short is equivalent to being 1-entangled.

  \begin{prop} Suppose that $L$ is a 2-entangled linear order. If $f:L\onto X$ and $X\subseteq L$ is uncountable then there is a countable $A\subseteq X$ so that $f^{-1}(x)=\{x\}$ for all $x\in X\setm A$.
   \end{prop}
   \begin{proof}
  Inductively define $a_\alpha,b_\alpha=f(a_\alpha)$ for $\alpha<\omg$ so that $b_\beta\in X\setm \{f(b_\alpha),b_\alpha,a_\alpha:\alpha<\beta\}$ and $a_\beta\in f^{-1}(b_\beta)\setm\{b_\beta\}$. It is straightforward to check that $a_\beta,b_\beta\notin \{a_\alpha,b_\alpha:\alpha<\beta\}$. 
  
  If the induction can go on for $\omg$ steps then $\{(a_\alpha,b_\alpha):\alpha<\omg\}$ violates that $L$ is 2-entangled. Hence, $f^{-1}(b)=\{b\}$ for all $b\in X\setm A $ for some $\beta<\omg$ where $A=\{f(b_\alpha),b_\alpha,a_\alpha:\alpha<\beta\}$.
\end{proof}
  

\begin{cor}
 If $L$ is uncountable and strongly surjective then $L$ is not 2-entangled.
\end{cor}\begin{proof}
          Suppose that $L$ is also 2-entangled and pick an $\aleph_1$-dense $X\subseteq L$ so that $L\setm X$ is uncountable. Suppose that $f:L\onto X$.  Let $A$ be countable so that $f^{-1}(x)=\{x\}$ for all $x\in X\setm A$. In particular $f(x)=x$ for all $x\in X\setm A$. Also, $f[L\setm X]\subseteq A$ and so there is $c<d\in  L\setm X$ so that $f(c)=f(d)\in A$. Now, there must be some $x\in X\setm A$ so that $c<x<d$; however, $f(x)=x\in A$ is a contradiction.
          
         \end{proof}

We don't know if a strongly surjective linear order can contain a 2-entangled set. \cite{ARSh} claims that $\textmd{MA}_{\aleph_1}$ is consistent with the statement that every uncountable set of reals contains an $\aleph_1$-dense, 2-entangled set. So we conjecture that there are no uncountable real suborders which are strongly surjective in this model of $\textmd{MA}_{\aleph_1}$.

Also, we don't know how to produce strongly surjective orders of size $>\aleph_2$, or if every strongly surjective (real) order necessarily contains a minimal suborder.

\section{When all uncountable, strongly surjective linear orders are Aronszajn}\label{sec:allAronszajn}

Now, we look at various models where every uncountable, strongly surjective linear order must be Aronszajn. It was proved in \cite{raphael} already that $\mf c<2^\kappa$ implies that there are no strongly surjective  suborders of $\mb R$ of size $\kappa$. 

First, we strengthen the above by proving the next  rather general result.

\begin{tlemma}\label{lm:ineq}
 Suppose that $L$ is strongly surjective, $K\subseteq L$ is arbitrary and $D\subseteq K$ is dense in $K$. Then $2^{|K\setm D|}\leq 2^{|D|}$. 
\end{tlemma}
\begin{proof}
If $L$ is strongly surjective then any left/right limit point of $L$ is the limit of a countable increasing/decreasing sequence. This implies that any  convex subset $C$ of $L$ is a countable union of intervals, and in turn, the family of all convex subsets $\mc C(L)$ has size $\leq |L|^\oo\leq \mf c$ (since the number of intervals is $|L|\leq \mf c$).
  
  Suppose that $D$ is dense in $K\subseteq L$. Consider two order preserving maps $f,g:L\to K$ such that $f^{-1}(q)=g^{-1}(q)\neq \emptyset$ for all $q\in D$. We claim that $\ran(f)=\ran(g)$. Suppose that $f(x)=y\in \ran(f)\setm D$. Hence $q< y$ implies $f^{-1}(q)< x$ (let us denote these $q$ with $D^-$) and $y< q$ implies $x< f^{-1}(q)$ (let us denote these $q$ with $D^+$) for all $q\in D$. As $D$ is dense, if $z\in K$ satisfies $D^-< z< D^+$  then $z=y$. Now, what could be $g(x)$?  $f^{-1}(q)< x$ implies that $\{q\}=g(f^{-1}(q))< g(x)$ for all $q\in D^-$ and  $x< f^{-1}(q)$ implies  $g(x)< \{q\}=g(f^{-1}(q))$ for all $q\in D^-$. Hence, by the previous observation, $y=g(x)\in \ran(g)$.
 	
 	Now, the number of choices for a sequence $(C_q)_{q\in D}$ where $C_q\in \mc C(L)$ is at most $|\mc C(L)|^{|D|}\leq \mf c^{|D|}=2^{|D|}$. So,  $2^{|D|}$ is an upper bound for the number of possible ranges for an order preserving map $f:L\to K$ with $D\subseteq \ran(f)$. Since $L$ is strongly surjective, this number is $2^{|K\setm D|}$ i.e. the number of subsets of $K$ containing $D$. So  $2^{|K\setm D|}\leq 2^{|D|}$ as desired.

\end{proof}

The following two corollaries immediately follow:

\begin{cor}\label{cor:noc}
 A strongly surjective linear order cannot contain real suborders of size $\mf c$. 
\end{cor}

We will later see that there could be strongly surjective linear orders of size $\mf c=\aleph_1$. 

\begin{cor}\label{cor:onlyA}
 $\mf c<2^{\aleph_1}$ implies that any uncountable, strongly surjective linear order is Aronszajn.
\end{cor}

P. Schlicht independently proved the above corollary using the stronger assumption of $\diamondsuit$ (personal communication).

\begin{cor}\label{cor:small}
 $\mf c=\aleph_2=2^{\aleph_1}$ implies that any uncountable, strongly surjective linear order has size $<\mf c$.
\end{cor}
\begin{proof}
 Suppose that $|L|=\mf c$ and $L$ is strongly surjective. Build a 2-branching partition tree $T$ for $L$ (see \cite[page 248]{stevotrees} and the discussion there for details on partition trees). The height of $T$ is at most $\omg$ and so there is a minimal $\beta<\omg$ such that $|T_\beta|=\aleph_2$ since $|T|=\aleph_2$. Pick one point from each convex set in $T_\beta$ to find a $K\subseteq L$ of size $\mf c=\aleph_2$ which has density $\leq |\bigcup_{\alpha<\beta}T_\alpha|\leq \aleph_1$. Now, $2^\mf c\leq 2^{\aleph_1}$ should hold by Lemma \ref{lm:ineq}     however this is clearly not possible since $2^{\aleph_1}=\mf c$.
\end{proof}

Next, let us show that  $\mf c=2^{\aleph_1}$ is also consistent with the statement that ``any uncountable, strongly surjective linear order is Aronszajn''. Let $\mb C_\kappa$ be the forcing adding $\kappa$-many Cohen reals i.e. $\mb C_\kappa=Fn(\kappa,2)$ the set of finite partial functions from $\kappa$ to 2.

  \begin{theorem} Suppose  GCH holds in $V$. Then $V^{\mb C_{\aleph_2}}$ is a model of   $\mf c=\aleph_2=2^{\aleph_1}$ and every uncountable, strongly surjective linear order is Aronszajn.
  \end{theorem}
	  \begin{proof}
	   
	      It is well known that $V^{\mb C_{\aleph_2}}\models \mf c=\aleph_2=2^{\aleph_1}$. So, our goal is to show that a strongly surjective linear order contains no uncountable separable suborders. 
	      

	   First, any strongly surjective  linear order $L$ in $V^{\mb C_{\aleph_2}}$ has size $<\aleph_2$ by Corollary \ref{cor:small} and in turn, $L$ appears in an intermediate model. Since $\mb C_{\aleph_2}=\mb C_{\aleph_2}*\mb C_{\aleph_1}$, it suffices to prove the following.
   
    \begin{tlemma}
     Suppose that $L=(\omega_1,\tri)$ is a linear order from the ground model $V$ so that $\oo$ is dense in some uncountable $L_0\subseteq L$. Then $L$ is not strongly surjective in $V^{\mb C_{\aleph_1}}$.
    \end{tlemma}
	   
	      \begin{proof}
	   Let $G\subseteq \mb C_{\aleph_1}$ be a $V$-generic filter; this gives a generic map $g:L_0 \to 2$ by $g=\cup G\uhr L_0$. Now, in $V[G]$, we define $$ K=\oo\cup g^{-1}(1)\subseteq L_0.$$ We will show that $L$ cannot be mapped onto $K$. To this end, suppose that $f:L\onto K$ is an order preserving surjection in $V[G]$. Find an appropriate countable $\nu\in \omg$ so that $V[G\uhr \nu]\models f(\xi_\ell)=\ell$ for an appropriate sequence  $(\xi_\ell)_{\ell\in \oo}\in \omg^\oo\cap V[G\uhr \nu] $. Now, note that $O_\alpha=\{\ell\in \oo: \ell\tri \alpha\}\in V$ for all $\alpha\in \omg$ since $L\in V$; hence $$A_\alpha=\{\xi_\ell:\ell\in O_\alpha\}, B_\alpha=\{\xi_\ell:\ell\in \oo\setm O_\alpha\}\in V[G\uhr \nu].$$

In particular, for  any $\alpha\in \omg$, it is decided in $V[G\uhr \nu]$ whether there is some $\xi\in \omg$ such that $A_\alpha\tri \xi\tri B_\alpha$.

\begin{claim} For any $\alpha\in\omg\setm \omega$, $g(\alpha)=1$ iff there is some $\xi\in \omg$ such that $A_\alpha\tri \xi\tri B_\alpha$.
 \end{claim}
 \begin{proof}
 If  $g(\alpha)=1$ then $\alpha\in K$ and hence any $\xi\in f^{-1}(\alpha)$ has to satisfy $A_\alpha\tri \xi\tri B_\alpha$. On the other hand, if $A_\alpha\tri \xi\tri B_\alpha$ for some $\xi\in \omg$ then $\alpha'=f(\xi)\in L_0$ must satisfy $O_\alpha\tri \alpha' \tri \oo\setm O_\alpha$. Since $\oo$ is dense in $L_0$, this clearly implies that $\alpha=\alpha'\in K$ and so $g(\alpha)=1$.
	   
	  \end{proof}   Note that the claim implies that $g$ can be defined in $V[G\uhr \nu]$ which is clearly not possible. This contradiction finishes the proof of the lemma.

	  \end{proof}

In turn, we proved the theorem.

   \end{proof}

  Now, a natural question is whether other classical models (Sacks, Miller, etc.) allow non Aronszajn strongly surjective linear orders. We claim that these models behave as the Cohen-model i.e. if there is a strongly surjective linear order, it has to be Aronszajn or countable. To prove this result, we employ the technique of parametrized weak diamonds \cite{DHM,osvaldo}.
  
  \begin{dfn}
   Let $\diamondsuit^{\omg}(2,=)$ denote the following statement: if $X$ is an $\omg$ set of ordinals and $F:\bigcup_{\delta<\omg} \delta^\delta\to 2$ so that $F\uhr \delta^\delta\in L(\mathbb R)[X]$ for all $\delta<\omg$, then there is a $g:\omg\to 2$ so that for all $f:\omg\to \omg$ the set $$\{\delta<\omg:f\uhr \delta\in \delta^\delta \textmd{ and } F(f\uhr \delta)=g(\delta)\}$$ is stationary.
  \end{dfn}
  
  Recall that $L(\mathbb R)$ is the class of sets constructible from $\mb R$ (in the sense of G\"odel) and $L(\mathbb R)[X]$ is the minimal model extending $L(\mathbb R)$ which contains $X$. See \cite[Chapter 13]{jech} for more details on constructibility.

  In Corollary \ref{cor:onlyA}, we saw that $\mf c<2^{\aleph_1}$ was used to deduce that all strongly surjective linear orders are Aronszajn. So, it is not a surprise that we turn to use weak diamonds: a celebrated result of K. Devlin and S. Shelah \cite{devlin} states that $\mf c<2^{\aleph_1}$ is equivalent to the above weak diamond statement if one drops the requirement of $F\uhr \delta^\delta\in L(\mathbb R)[X]$ i.e. that we only guess constructible functions $F$. However, as shown by O. Guzman and M. Hru\v s\'ak \cite{osvaldo},  $\diamondsuit^{\omg}(2,=)$ suffices to prove many classical consequences of $\mf c<2^{\aleph_1}$  e.g. the failure of $\textmd{BA}_{\aleph_1}$.
  
  Now, the advantage of $\diamondsuit^{\omg}(2,=)$ over $\mf c<2^{\aleph_1}$ is clear from the following theorem.
  
  \begin{theorem}[\cite{osvaldo}]
   $\diamondsuit^{\omg}(2,=)$ holds in all models resulting from a length $\aleph_2$ countable support iteration of a single Suslin-definable poset $\mb P$ which is homogeneous (i.e. $\mb P\equiv \{0,1\}\times \mb P$) and proper. 
  \end{theorem}

  The classical tree forcings (Sacks, Miller, etc.) satisfy all these requirements. Also, in any such canonical model $\mf c=2^{\aleph_1}=\aleph_2$. In turn, strongly surjective linear orders must have size $\leq \aleph_1$ by Corollary \ref{cor:small}. So, our final aim in this section is to prove the following theorem.
  
  \begin{theorem}
$\diamondsuit^{\omg}(2,=)$ implies that all strongly surjective linear orders of size $\aleph_1$ are Aronszajn.   
  \end{theorem}
 \begin{proof}
  Suppose that $L=(\omg,\tri)$ is a linear order and $A=\{a_\alpha:\alpha<\omg\}\subseteq L$ is an uncountable suborder with a countable set $D$ which is dense in $A$. We aim to find $B\subseteq A$ so that there is no $\tri$-preserving $f:\omg\onto B$.
  
  Let us define a function $F:\bigcup_{\delta<\omg} \delta^\delta\to 2$ using $\tri$ and $A$ as parameters and make sure that $F\uhr \delta^\delta\in L(\mb R)[A,\tri]$ as desired; it is standard to code $A$ and $\tri$ into a single set of ordinals of size $\aleph_1$ so we omit these details.
  
  Fix $\delta<\omg$ and $f:\delta\to \delta$. Also, let $\alpha_n=\alpha^\delta_n\tri \beta_n=\beta_n^\delta\in D$ so that $$\{a_\delta\}=A\cap \bigcap_{n\in \oo}[\alpha_n,\beta_n]_{\tri}.$$ We let $F(f)=1$ iff
  \begin{enumerate}[(a)]
   \item $D\subseteq \ran (f)$, 
   \item $f$ is $\tri$-preserving and
   \item the set $$\bigcap_{n\in \oo}\{[\xi,\zeta]_{\tri}:\xi\in f^{-1}(\alpha_n), \zeta\in f^{-1}(\beta_n)\}$$ is empty.
  \end{enumerate} In any other case, we let $F(f)=0$.   Note that (a) and (b) are clearly Borel conditions but we do use $A,\tri$ as parameters in condition (c).
  
  Now,    $\diamondsuit^{\omg}(2,=)$ hands us some function $g:\omg\to 2$. We use $g$ to define a subset of $A$ as follows:
  
  $$B=D\cup \{a_\delta:\delta\in \omg \textmd{ and } g(\delta)=1\}.$$
  
  We claim that there is no order preserving map $f:\omg \onto B$. Otherwise, the set $$S=\{\delta<\omg:f\uhr \delta\in \delta^\delta \textmd{ and } F(f\uhr \delta)=g(\delta)\}$$ is stationary. Pick a large enough $\delta\in S\setm \oo$ so that $D\subseteq \ran(f\uhr \delta)$.
  
  Now, there are two possible cases: first, suppose that $g(\delta)=1$. Then $a_\delta\in B$ so there is some $\nu\in f^{-1}(a_\delta)$. Note that $\alpha_n\tri a_\delta\tri \beta_n$ implies that $\xi\tri \nu\tri \zeta$ for all $\xi\in f^{-1}(\alpha_n), \zeta\in f^{-1}(\beta_n)$ since $f$ is $\tri$-preserving. So $\nu\in  \bigcap_{n\in \oo}\{[\xi,\zeta]_{\tri}:\xi\in f^{-1}(\alpha_n), \zeta\in f^{-1}(\beta_n)\}$ which means that $F(f\uhr \delta)=0$. However, $\delta\in S$ implies that $0=F(f\uhr \delta)=g(\delta)=1$ a contradiction.
  
  Second, suppose that $g(\delta)=0=F(f\uhr \delta)$. In particular $a_\delta\notin B$. However, $0=F(f\uhr \delta)$ implies that $$\bigcap_{n\in \oo}\{[\xi,\zeta]_{\tri}:\xi\in f^{-1}(\alpha_n), \zeta\in f^{-1}(\beta_n)\}\neq \emptyset$$ since conditions (a) and (b) are satisfied. Pick any $\nu \in \bigcap_{n\in \oo}\{[\xi,\zeta]_{\tri}:\xi\in f^{-1}(\alpha_n), \zeta\in f^{-1}(\beta_n)\}$ and look at the image $f(\nu)\in B$. We must have $\alpha_n\tri f(\nu)\tri \beta_n$ for all $n\in \oo$ however the only element of $A$ which can satisfy this is $a_\delta$ (by the choice of $\alpha_n,\beta_n$). So $f(\nu)=a_\delta\in B$, a contradiction again.
  

 \end{proof}

  It is not clear at this point if there are any uncountable strongly surjective linear orders in the Cohen-or the above canonical models.



\section{A minimal, strongly surjective Suslin tree}\label{sec:diamond}

As the title suggests, our next goal is to construct a lexicographically ordered Suslin tree which is strongly surjective and minimal. A construction for a minimal Suslin tree was presented first by J. Baumgartner \cite{baum}. However, recently Hossein Lamei Ramandi pointed out that Baumgartner’s crucial \cite[Lemma 4.14]{baum} is unrepairably flawed. We hope to present a correct and complete proof now.

We will start by a few necessary definitions: if $T$ is a tree then let $s^\downarrow$ denote $\{t\in T:t\leq s\}$ and $s^{\uparrow}=\{t\in T:t>s\}$. If $\xi<\Ht(s)$ then let $s\uhr \xi$ denote the unique $t< s$ with height $\xi$.

A tree $T$ is \emph{$\omega$-branching} if $\textmd{br}(t):=\{s\in T: s^\downarrow\setm\{s\}=t^\downarrow\setm \{t\}\}$ is countably infinite
 for all $t\in T_\alpha$. Note that {$\omega$-branching} trees branch at limit levels as well i.e. $s\uhr \zeta=t\uhr \zeta$ for all $\zeta<\xi$ does not imply that $s\uhr \xi=t\uhr \xi$.

 Now, suppose that apart from the tree order there is a linear order $\tri$ on each set of the form $\textmd{br}(t)$ so that $$(\textmd{br}(t),\tri)\equiv \mb Q$$ for all $t\in T$. Then we say that $T$ is \emph{doubly ordered}.

Next, given a doubly ordered $T$, we can define the corresponding lexicographic order $<_{\lex}$ which linearly orders $T$. That is, let $s<_{\lex} t$ iff $s<t$ in the tree order or $s\uhr \xi\tri t\uhr \xi$ where $\xi$ is minimal so that $s\uhr \xi\neq t\uhr \xi$


A \emph{double isomorphism} between two doubly ordered trees is a bijection which preserves both the tree and lexicographic order.

We say that $Y\subseteq T$ is \emph{large} if $Y$ is cofinal in $T[Y]=\{t\in T: t\leq y$ or $y\leq t$ for some $y\in Y \}$; note that a large subset is not necessarily downward closed. 

Finally recall that  $\diamondsuit^+$ is the following statement: there is a sequence $(\mc S_\alpha)_{\alpha\in \omg}$ so that $\mc S_\alpha$ is countable and for any $Y\subseteq \omg$ there is a club subset $C$ of $\omg$ so that $Y\cap \alpha,C\cap \alpha\in \mc S_\alpha$ for all $\alpha\in C$.  $\diamondsuit^+$ is a well known consequence of $V=L$.

We will prove: 

\begin{theorem}\label{suslinthm} Under $\diamondsuit^+$, there is an $\oo$-branching doubly ordered Suslin tree $T$ so that $T$ and $T\uhr Y$ (the order from $T$ restricted to $Y$) are doubly isomorphic whenever $Y\subseteq T$ is large and $\min Y$ is isomorphic to $\mb Q$ .
\end{theorem}

At the end of the section (in Corollary \ref{cor:strongsuslin}), we show  that the above tree is strongly surjective; a similar argument is presented in \cite[Theorem 5.18]{raphael} already but we include a proof as well for the sake of completeness. 

\smallskip

Our construction will be a modification of \cite[Theorem 2.3]{hns} where a Suslin tree $T$ is constructed such that $T$ and $T\uhr Y$ are tree isomorphic whenever $Y\subseteq T$ is large. Actually, we believe that the tree constructed in \cite[Theorem 2.3]{hns} can be doubly ordered to satisfy our requirements but we repeat the construction anyways. The proof that large subsets of $T$ are actually doubly isomorphic to $T$ and not just isomorphic as trees requires extra thought anyways.  It seems that the authors of \cite{hns} were not aware of Baumgartner's \cite{baum} at the time of publishing.

 We start by stating the relevant main lemma \cite[Lemma 2.10]{hns} in a strong form. Given an $\omega$-branching tree $T$ and a set $B$ of  unbounded branches in $T$, one can naturally define a tree $\tradd{T}{B}$ which is an end extension of $T$ as follows: $\tradd{T}{B}$ is defined on the set $T\cup \{(b,i):b\in B, i<\oo\}$ and we require that $s<(b,i)$ for all $s\in b\in B$ and $i<\oo$. Note that $\tradd{T}{B}$ is $\omega$-branching if $T$ was $\omega$-branching and $\Ht(\tradd{T}{B})=\Ht(T)+1$.
 
 The next crucial lemma says that such an extension can be done while preserving certain antichains and isomorphisms.

\begin{tlemma}\label{extlemma}
Suppose that $T$ is an $\oo$-branching tree of limit height $\alpha\in \omg$ and $\mc N$ is a countable set. Then there is a countable set $B=B(T,\mc N)$ of unbounded branches in $T$ so that the $\oo$-branching end extension $\tradd{T}{B}$ of $T$ of height $\alpha+1$ satisfies the following:
\begin{enumerate}
 \item if  $A\in \mc N$ is a  maximal antichain in $T$ then $A$ is still maximal in $\tradd{T}{B}$, and

 \item  if $f\in \mc N$ is a tree isomorphism between large subsets of $T$ then
 
\begin{enumerate}
 \item $b\in B$ implies $f[b]\in B$, and 
 \item the assignment $(b,i)\mapsto (f[b],i)$ extends $f$ to an isomorphism between large subsets of $\tradd{T}{B}$.
\end{enumerate}\end{enumerate}

\end{tlemma}

The authors of \cite{hns} actually state their Lemma 2.10 with less details about the nature of the extension but they prove exactly what we wrote above; the proof is rather delicate and requires careful thought. We omit reproducing this argument here but highly recommend that the interested reader studies the details.

Suppose we are in the setting of Lemma \ref{extlemma} and $T$ is actually doubly ordered. We can extend this double order to $\tradd{T}{B}$ as follows: order $\{(b,i): i<\oo\}$ as $\mb Q$ for each $b\in B$ by a fixed bijection $\varphi:\oo\to \mb Q$. The particular way in which the trees and isomorphisms in Lemma \ref{extlemma}  were extended allows double isomorphism to extend too; this will be the content of the next lemma. Let $\mc L(T,S)$ be the set of double isomorphisms between large subsets of $T$ and $S$.

\begin{tlemma}\label{extension} Suppose that $T,\mc N$ and $B$ are as in Lemma \ref{extlemma}, $T$ is doubly ordered and $T^{*}$ is a doubly ordered end extension of $\tradd{T}{B}$. 

If $f\in \mc N\cap \mc L(T,T)$ and  $g$ is a bijection between two large subsets of $\tradd{T}{B}$ and $T^{*}$ so that 
\begin{enumerate}[(1)]
 \item $f\subseteq g$,
 \item $t\leq_{T^{*}}g(b,i)$ for each $t\in f[b],b\in B$, and $i<\oo$, and
 \item  $g\uhr\{(b,i):i<\oo\}$ is $<_{\lex}$-preserving for each $b\in B$.
\end{enumerate}

then $g\in \mc L(\tradd{T}{B}, T^*)$ i.e. $g$ is a double isomorphism as well.
 \end{tlemma}
 
 \begin{figure}[H]\centering
 \psscalebox{0.7} 
{
\begin{pspicture}(0,-2.8047407)(18.5,2.8047407)
\rput{102.81887}(9.697806,-4.889845){\psarc[linecolor=black, linewidth=0.04, dimen=outer](6.8,1.424599){1.06}{242.44719}{270.0}}
\rput{-74.724846}(1.5866416,4.9272623){\psarc[linecolor=black, linewidth=0.04, dimen=outer](4.02,1.424599){1.06}{242.44719}{270.0}}
\psbezier[linecolor=black, linewidth=0.04](3.0,1.024599)(3.0,0.224599)(3.8,-2.175401)(5.4,-2.1754010009765623)(7.0,-2.175401)(7.8,0.224599)(7.8,1.024599)
\psline[linecolor=black, linewidth=0.04](1.34,1.404599)(8.6,1.424599)
\psline[linecolor=black, linewidth=0.04](1.28,1.024599)(2.6525,-2.575401)(7.2275,-2.575401)(8.6,1.024599)(1.28,1.024599)
\psbezier[linecolor=black, linewidth=0.04](11.6,1.064599)(11.72,0.264599)(12.4,-2.135401)(14.0,-2.1354010009765627)(15.6,-2.135401)(16.3,0.284599)(16.4,1.064599)
\psline[linecolor=black, linewidth=0.04](10.8,1.464599)(17.2,1.464599)
\psline[linecolor=black, linewidth=0.04](10.8,1.064599)(12.0,-2.535401)(16.0,-2.535401)(17.2,1.064599)(10.8,1.064599)
\psline[linecolor=black, linewidth=0.04, linestyle=dashed, dash=0.17638889cm 0.10583334cm, arrowsize=0.05291667cm 2.0,arrowlength=1.4,arrowinset=0.0]{->}(8.6,-0.975401)(10.6,-0.975401)
\psline[linecolor=black, linewidth=0.04, linestyle=dashed, dash=0.17638889cm 0.10583334cm, arrowsize=0.05291667cm 2.0,arrowlength=1.4,arrowinset=0.0]{->}(8.78,2.224599)(10.34,2.224599)
\rput[bl](9.52,-0.635401){$f$}
\rput[bl](9.5,2.544599){$g$}
\psline[linecolor=black, linewidth=0.04](10.82,1.444599)(10.52,2.784599)(17.6,2.744599)(17.18,1.464599)
\rput{-69.4407}(6.888257,14.6295595){\psarc[linecolor=black, linewidth=0.04, dimen=outer](14.0,2.344599){2.6}{239.82298}{270.0}}
\rput{103.959206}(19.552382,-10.958009){\psarc[linecolor=black, linewidth=0.04, dimen=outer](14.06,2.164599){2.48}{239.82298}{270.0}}
\psline[linecolor=black, linewidth=0.04](13.84,-2.715401)(14.7,-0.555401)(14.1,0.204599)(14.38,0.924599)
\psline[linecolor=black, linewidth=0.04](5.84,-2.795401)(4.9,-1.015401)(5.16,-0.195401)(4.62,0.844599)
\psline[linecolor=black, linewidth=0.04](5.24,-1.715401)(6.44,-0.235401)(6.58,0.844599)
\rput[bl](4.1,0.264599){$b$}
\psdots[linecolor=black, dotsize=0.16](4.58,1.404599)
\psdots[linecolor=black, dotsize=0.16](14.52,2.044599)
\rput[bl](13.22,0.164599){$f[b]$}
\rput[bl](13.38,2.044599){$g(b,i)$}
\rput[bl](4.14,1.784599){$(b,i)$}
\rput[bl](6.22,1.964599){$\dom(g)$}
\rput[bl](2.6,-1.615401){$\dom(f)$}
\rput[bl](17.38,-1.135401){\Large{$T$}}
\rput[bl](18.12,1.944599){\Large{$T^*$}}
\rput[bl](0.5,-1.475401){\Large{$T$}}
\rput[bl](0.0,1.784599){\Large{$\tradd{T}{B}$}}
\end{pspicture}
}\caption{The setting of Lemma \ref{extension}}
\label{fig:lemma}
  
 \end{figure}
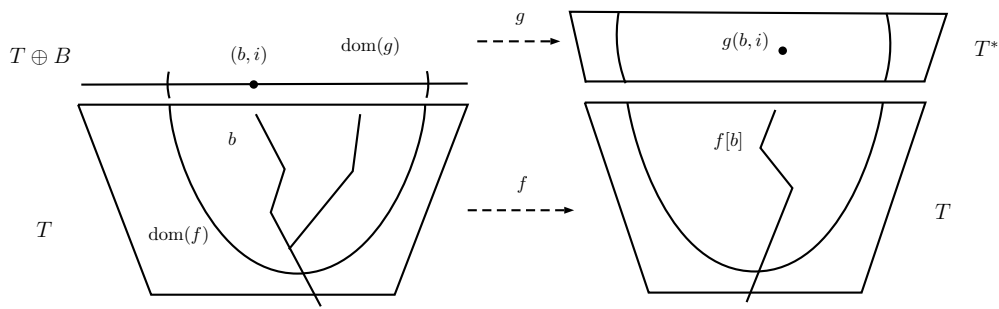

\begin{proof}
 
 Suppose that $g$ is as above and $s<_{T\oplus B}t\in \dom(g)$. Note that $g$ must be a tree isomorphism by (1) and (2) and Lemma \ref{extlemma}.
 
 First, if $t=(b,i)$ then $f(s)=g(s)\in f[b]<_{T^{*}}g(b,i)$ by (2). 
 
 Second, suppose that $s\uhr \xi<_{\lex}t\uhr \xi$ where $\xi$ is minimal so that $s\uhr \xi\neq t\uhr \xi$; if $\xi=\alpha$ then (3) implies that $g(s)<_{\lex} g(t)$ as well.
 If $\xi<\alpha$ then $f(s\uhr \zeta)=f(t\uhr \zeta)$ for all $\zeta<\xi$ and  $f(s\uhr \xi)<_{\lex}f(t\uhr \xi)$  since $f$ was assumed to be a double isomorphism. Now, $f(s\uhr \xi)=g(s\uhr \xi)<_{T^{*}} g(s)$ and $f(t\uhr \xi)=g(t\uhr \xi)<_{T^{*}} g(t)$ so $g(s)<_{\lex} g(t)$ as desired.

\end{proof}

Finally, we need:

\begin{tlemma}\label{bflemma} Any two countable, $\oo$-branching doubly ordered trees $S,T$ of the same countable limit height are doubly isomorphic.
 \end{tlemma}
 \begin{proof}
  
 This is a standard back-and-forth argument. Alternatively, consider the partial order $Q=Q_{S,T}$ of finite partial double-isomorphisms and show that the set of $q\in Q$ so that $t\in \dom(q),s\in\ran(q)$ is dense open for any $s\in S,t\in T$. Now apply the Baire-category theorem (or the Rasiowa-Sikorski lemma) to find a filter $G\subseteq Q$ meeting all these countably many dense sets. The map $\cup G:S\to T$ is the desired double isomorphism.
 
 \end{proof}

 
 We are ready to construct our Suslin tree now:

\begin{proof}[of Theorem \ref{suslinthm}]
 
 Let $(\mc S_\alpha)_{\alpha<\omg}$ denote the $\diamondsuit^+$ sequence. Take an increasing sequence of elementary submodels $(N_\alpha)_{\alpha<\omg}$ of $(H(\aleph_2),\in,\prec)$ so that $\mc S_\beta, (N_\alpha)_{\alpha<\beta}\in N_\beta$ for all $\beta<\omg$ (note that the sequence of models is not continuous). Here $\prec$ denotes an arbitrary well order of $ H(\aleph_2)$. 
 
 We construct an increasing sequence of $\oo$-branching, doubly ordered trees $(T^\alpha)_{\alpha\leq \omg}$ on subsets of $\omg$ so that 
 \begin{enumerate}[$(i)_\beta$]
 
 \item $(T^\alpha)_{\alpha\leq \beta}\in N_\beta$, 
 \item$\Ht(T^\beta)=\beta$, and
 \item the sequence of trees $(T^\alpha)_{\alpha\leq \beta}$ can be uniquely recovered from the sequence $(N_\alpha)_{\alpha<\beta}$,
 \item $T^\beta$ is an end extension of $T^\alpha$ for all $\alpha<\beta\leq \omg$.
  \end{enumerate}
  
  These properties will be ensured by making uniform choices (using $\prec$) when extending the trees.

 Suppose $(T^\alpha)_{\alpha< \beta}$ is constructed so that $(i)_{\beta'}-(iv)_{\beta'}$ holds for all $\beta'<\beta$. If $\beta$ is a limit then we let $T^\beta=\bigcup \{T^\alpha:\alpha< \beta\}$. If $\beta=\alpha+1$ and $\alpha$ is a successor then we take the $\prec$-minimal end extension of $T^\alpha$ in $N_\beta$ that satisfies our requirements (i) and (ii).  
 
 Finally, if $\beta=\alpha+1$ and $\alpha$ is a limit then we apply Lemma \ref{extlemma} to $T^\alpha$ and $\mc N=N_\beta$ and define $T^\beta=\tradd{T^\alpha}{B(T^\alpha,N_\alpha)}$. The set of branches $B_\alpha=B(T^\alpha,N_\alpha)$ is chosen $\prec$-minimal in $N_{\beta+1}$ which again ensures $(i)_\beta$ and $(iii)_\beta$.
 
 The tree $T=T^\omg$ we constructed is Suslin. Indeed, suppose that $A$ is a maximal antichain. Then there is an $\beta<\omg$ so that $A\cap \beta\in \mc S_\beta\subseteq N_\beta$ and $A\cap \beta$ is maximal in $T_{<\beta}=T^\beta$. Recall that we applied Lemma \ref{extlemma} with $\mc N=N_\beta$ to construct $T_{\leq\beta}=T^{\beta+1}$. So we preserved $A\cap \beta$ as a maximal antichain in $T_{\leq\beta}$ and hence in $T$. So $A=A\cap \beta$ is countable.

 Now suppose that $Y$ is large. Let $C$ be a club subset of $\omg$ so that $\gamma\in C$ implies that $C\cap \gamma, Y\cap \gamma \in N_\gamma$ and $T_{<\gamma}=T\cap \gamma$ and $Y\cap \gamma=Y_{<\gamma}$ is large in $T_{<\gamma}$. Let $C=\{\gamma_\nu:\nu<\omg\}$ denote the increasing enumeration of $C$.
 
 We inductively construct maps $(\pi_\nu)_{\nu\leq \omg}$  along the club $C$ so that 
 \begin{enumerate}
  \item $\pi_\nu:T_{<\gamma_\nu}\to Y\cap T_{<\gamma_\nu}$ is a double isomorphism and $\pi_\nu\in N_{\gamma_\mu}$,
  \item $\pi_\nu\subseteq \pi_{\nu'}$ for all $\nu<\nu'$, and
  \item the sequence of maps $(\pi_\mu)_{\mu\leq \nu}$ can be uniquely recovered from the sequence $(N_{\gamma_\mu})_{\mu<\nu}$,
 \end{enumerate}

The initial step can be achieved since $\min(Y)$ is isomorphic to the rational numbers. At limit steps, we take unions and our construction guarantees that the double isomorphism $\pi_{<\nu}=\cup_{\mu<\nu} \pi_
 \mu:T_{<\gamma_{\nu}}\to Y_{<\gamma_{\nu}}$ is in $N_{\gamma_\nu}$. Indeed, the model $N_{\gamma_\nu}$ contains enough information (the $\diamondsuit^+$ sequence and the club up to $\gamma_\nu$) so that we can recover the construction of the sequence $(\pi_\mu)_{\mu<\nu}$ working in $N_{\gamma_\nu}$, and hence define the union as well. This is where the full force of $\diamondsuit^+$ is applied.
 
 Now, at successors of limit stages we do the following: recall that in our construction of $T_{\leq \gamma_\nu}$ from $T_{<\gamma_\nu}$ we added nodes $\{(b,i):b\in B_{\gamma_\nu},i\in \oo\}$ for a countable set of unbounded chains $B_{\gamma_\nu}$. So, we will define $\rho:T_{\leq\gamma_{\nu}}\to Y_{\leq\gamma_{\nu}}$ in $N_{\gamma_\nu}$ such that 
 \begin{enumerate}
 \item $\pi_{<\nu}\subseteq \rho$,
 \item $\pi_{<\nu}[b]\leq_{T}\rho(b,i)$, and
 \item  $\rho\uhr\{(b,i):i<\oo\}$ is $<_{\lex}$-preserving for each $b\in B_{\gamma_\nu}$.
\end{enumerate}
 If we can do this then $\rho$ must be a double isomorphism by Lemma \ref{extension}. Next, we can extend $\rho $ to $\pi_\nu$ in $N_{\gamma_{\nu+1}}$  further by an isomorphism $T_{<\gamma_{\nu+1}}\setm T_{\leq\gamma_{\nu}}\to Y_{<\gamma_{\nu+1}}\setm Y_{\leq\gamma_{\nu}}$; this is possible by Lemma \ref{bflemma}.
 
Lets see how to construct $\rho$: first, recall that $(\pi_{<\nu}[b],i)$ is a node of $T_{\gamma_\nu}$ which is above the set $\pi_{<\nu}[b]=\{\pi_{<\nu}(t):t\in b\}$ for all $b\in B_{\gamma_\nu}$.  Now, consider $R_b=\{y\in Y_{\gamma_\nu}: \pi_{<\nu}[b]<y\}$.

\begin{obs}
$R_b$ is isomorphic to $\mb Q$ in the lexicographic order.  
\end{obs}
\begin{proof} First, note that $R_b=\bigcup\{R_{b,i}:i<\oo\}$ where $R_{b,i}=\{y\in Y_{\gamma_\nu}:(\pi_{<\nu}[b],i)\leq y\}$. If $(\pi_{<\nu}[b],i)\in Y$ then $R_{b,i}=\{(\pi_{<\nu}[b],i)\}$; otherwise, $R_{b,i}$ is isomorphic to $\mb Q$ since $Y$ was large so for each successor $x$ of $(\pi_{<\nu}[b],i)$ there is $x\leq y\in Y_{\gamma_\nu}$. This clearly implies that $R_b$ is isomorphic to $\mb Q$ as well.
\end{proof}

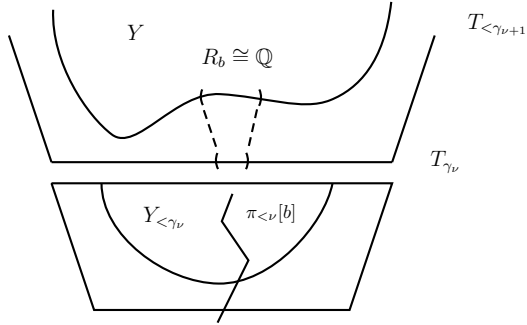
\begin{figure}[H]
\centering
 \psscalebox{0.7} 
{
\begin{pspicture}(0,-3.0436158)(9.638973,3.0436158)
\psline[linecolor=black, linewidth=0.04](0.8189737,-0.39479324)(1.6189737,-2.7947931)(6.418974,-2.7947931)(7.2189736,-0.39479324)(0.8189737,-0.39479324)
\psline[linecolor=black, linewidth=0.04](0.8189737,0.0052067568)(0.018973693,2.4052067)
\psline[linecolor=black, linewidth=0.04](0.8189737,0.0052067568)(7.2189736,0.0052067568)
\psline[linecolor=black, linewidth=0.04](7.2189736,0.0052067568)(8.018973,2.4052067)
\psbezier[linecolor=black, linewidth=0.04](0.8500615,2.8924994)(0.81038874,1.8837733)(1.4927074,0.87571114)(1.955026,0.5276489698311889)(2.4173448,0.17958678)(3.1434827,1.3415129)(3.9877927,1.2941146)(4.832103,1.2467161)(5.5290956,0.9688601)(6.0552406,1.1797621)(6.5813856,1.3906641)(7.072996,1.8954555)(7.2022233,3.0413747)
\psbezier[linecolor=black, linewidth=0.04](1.7589737,-0.39479324)(1.7572922,-1.3870184)(2.9590788,-2.280302)(3.9589736,-2.2947932434082032)(4.9588685,-2.3092844)(5.946005,-1.1795824)(6.0989738,-0.39479324)
\psline[linecolor=black, linewidth=0.04](3.9389737,-3.0347931)(4.518974,-1.8547932)(4.018974,-1.1147933)(4.2189736,-0.59479326)
\rput[bl](4.478974,-1.1347933){$\pi_{<\nu}[b]$}
\rput[bl](2.5389738,-1.2547933){\Large{$Y_{<\gamma_\nu}$}}
\rput[bl](7.938974,-0.19479324){\Large{$T_{\gamma_\nu}$}}
\rput[bl](8.638973,2.4052067){\Large{$T_{<\gamma_{\nu+1}}$}}
\rput[bl](2.2389736,2.3252068){\Large{$Y$}}
\rput{-62.509647}(1.0324594,4.197685){\psarc[linecolor=black, linewidth=0.04, dimen=outer](3.9743583,1.2482836){0.35076922}{219.87552}{270.0}}
\rput{114.14107}(7.3101516,-2.2686858){\psarc[linecolor=black, linewidth=0.04, dimen=outer](4.389743,1.2328991){0.36615384}{219.87552}{270.0}}
\rput[bl](3.6343584,1.8052068){\Large{$R_b\cong \mb Q$}}
\rput{-62.509647}(2.2611682,3.806765){\psarc[linecolor=black, linewidth=0.04, dimen=outer](4.266666,0.04059137){0.35846153}{219.87552}{270.0}}
\rput{114.14107}(5.868275,-3.7348418){\psarc[linecolor=black, linewidth=0.04, dimen=outer](4.143589,0.032899063){0.36615384}{219.87552}{270.0}}
\psline[linecolor=black, linewidth=0.04, linestyle=dashed, dash=0.17638889cm 0.10583334cm](3.9112813,0.29751444)(3.6805122,0.98982215)
\psline[linecolor=black, linewidth=0.04, linestyle=dashed, dash=0.17638889cm 0.10583334cm](4.526666,0.31289905)(4.680512,1.0052067)
\end{pspicture}
}
\caption{Extending $\pi_{<\nu}$}
\label{fig:extdiag}
\end{figure}

Hence, we can take a sequence $(\rho_b)_{b\in B_{\gamma_\nu}}$ so that $\rho_b:\{(b,i):i<\oo\}\to R_b$ is an order isomorphism; we make sure to choose this sequence $\prec$-minimal and in $N_{\gamma_\nu}$. We let $\rho=\pi_{<\nu}\cup\{\rho_b:b\in B_{\gamma_\nu}\}$ and $\rho\in N_{\gamma_\nu}$ is as desired.

At successors of successors, we simply apply Lemma \ref{bflemma}. Finally, $\pi=\pi_\omg$ is a double isomorphism between $T$ and $Y$.

\end{proof}

\begin{cor}\label{cor:strongsuslin} Under $\diamondsuit^+$, there is a lexicographically ordered Suslin tree which is strongly surjective and minimal. 
\end{cor}
\begin{proof}
 We will show that the doubly ordered Suslin tree $T$ from Theorem \ref{suslinthm} with the lexicographic order is strongly surjective; $T$ is clearly minimal since each uncountable subset of a Suslin tree contains a large subset (see Observation \ref{suslinobs} below).

 \begin{claim}\label{timesQclaim}
  \begin{enumerate}[(1)]
   \item $T$ and $T\times \mb Q$ are isomorphic.
   \item If $X\subseteq T$ then there is a countable $A\subseteq T$ so that $X=\sum_{a\in A} K_a$ and for any $a\in A$ either
   \begin{enumerate}[(a)]
    \item $K_a$ is a singleton, or
    \item  $K_a$ isomorphic to $T$ and then there is a dense interval $I$ of $A$ around $a$ so that $a'\in I$ implies $K_{a'}$ is a copy of $T$ as well.
   \end{enumerate}
   \end{enumerate}
  Moreover, any linear order $T$ with properties (1) and (2) is strongly surjective. 
  
 \end{claim}
 
 \begin{proof}
 (1) First, write $\mb Q$ as $\sum_{q\in \mb Q}I_q$ so that $I_q\subseteq \mb Q$ is isomorphic to $\mb Q$. Take an isomorphism $\pi:\mb Q\to \min T$. Note that $T[\pi[I_q]]$ and $T$ are isomorphic since $T[\pi[I_q]]$ is a large subset of $T$. Hence, the decomposition $T=\sum_{q\in \mb Q}T[\pi[I_q]]$ witnesses that $T$ and $T\times \mb Q$ are isomorphic.
 
 (2) We start by a standard observation:
  \begin{obs}\label{suslinobs}
   If $X\subseteq T$ is uncountable then $X\cap a^{\uparrow}$ is large for some $a\in X$.
  \end{obs}
\begin{proof} 
 This is elementary Suslin tree combinatorics: we can show that $X\cap a^{\uparrow}$ is cofinal in $a^{\uparrow}$. Otherwise, for all $a\in X$ we can find  $t_a\geq a$ so that $\{t\in X:t>t_a\}$ is empty . Select a maximal antichain $W$ from $\{t_a:a\in X\}$ and pick any $b\in X$ so that $\Ht(b)>\Ht[W]$. Then $b\leq t_b$ and $t_a\leq t_b$ for some $t_a\in W$. This, however, implies that $t_a<b\in X$ which contradicts the choice of $t_a$.
\end{proof}

 Now, we will find a countable $A_0$ so that $X=\sum_{a\in A_0} K_a$ where either $K_a$ is a singleton or has order type $T$. (2) clearly follows, since each copy of $T$ is actually isomorphic to $T\times \mb Q$ by (1).
 
 In order to find $A_0$, let $S$ denote the minimal elements of the set $$\{a\in T: X\cap a^{\uparrow}\textmd{ is large}\}.$$ $S$ is an antichain so $|S|\leq \oo$.  Furthermore, the set  $R=X\setm \bigcup_{a\in S}(X\cap a^{\uparrow})$ is countable. Indeed, this follows from Observation \ref{suslinobs}.

 
 
Finally, let $A_0=S\cup R$ and  
 \[
     K_a= \begin{cases}
 		     \{a\}, & \text{for } a\in R, \text{ and} \\
 		     X\cap a^{\uparrow},  & \text{for } a\in S.
         \end{cases}
   \] 
 
 Now, if $K_a=X\cap a^{\uparrow}$ then $K_a$ is isomorphic to $T$ so $X=\sum_{a\in A_0}K_a$ as desired. 

 Finally, fix a $T$ with properties (1) and (2). Let  $X\subseteq T$ and we will show that $T\onto X$. Write  $X=\sum_{a\in A} K_a$ as in (2); our aim is to find an $f:T\times \mb Q\onto X$, and, since $T\times \mb Q$ and $T$ are isomorphic by (1), this will finish the proof.
 
 Let  \[
     Q_a= \begin{cases}
 		     \{a\}, & \text{if } K_a \text{ is a copy of }T, \text{ and} \\
 		     \mb Q,  & \text{otherwise} 
         \end{cases}
   \] for each $a\in A$.

 The choice of $A$ guarantees that $\sum_{a\in A}Q_a$ is countable and dense without endpoints, so it is isomorphic to $\mb Q$. The function $$h:\mb Q\cong \sum_{a\in A}Q_a\onto A$$ that maps $Q_a$ onto $\{a\}$ is order preserving and satisfies that $|h^{-1}(a)|=1$ iff $K_a$ is a copy of $T$.
 

 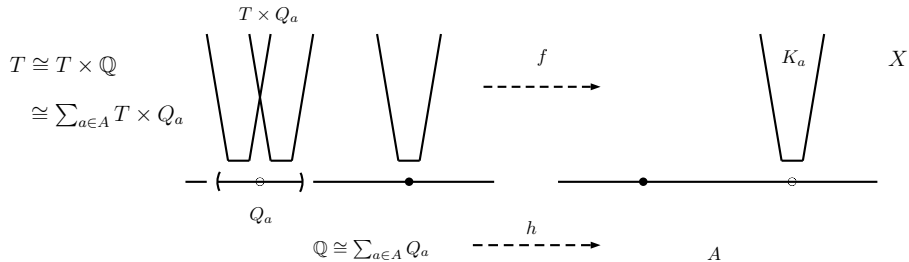
\begin{figure}[H]\centering
  \psscalebox{0.7} 
{
\begin{pspicture}(0,-2.2974513)(15.22,2.2974513)
\psdots[linecolor=black, dotsize=0.16](5.8,-1.0125488)
\psdots[linecolor=black, dotstyle=o, dotsize=0.16, fillcolor=white](3.0,-1.0125488)
\psline[linecolor=black, linewidth=0.04](4.0,-1.0125488)(7.4,-1.0125488)
\psdots[linecolor=black, dotstyle=o, dotsize=0.16, fillcolor=white](13.0,-1.0125488)
\psdots[linecolor=black, dotsize=0.16](10.2,-1.0125488)
\psline[linecolor=black, linewidth=0.04](8.6,-1.0125488)(14.6,-1.0125488)
\rput[bl](11.4,-2.5){$A$}
\psline[linecolor=black, linewidth=0.04](12.8,-0.6125488)(13.2,-0.6125488)
\psline[linecolor=black, linewidth=0.04](13.2,-0.6125488)(13.6,1.7874511)
\psline[linecolor=black, linewidth=0.04](12.8,-0.6125488)(12.4,1.7874511)
\psline[linecolor=black, linewidth=0.04](5.6,-0.6125488)(6.0,-0.6125488)
\psline[linecolor=black, linewidth=0.04](6.0,-0.6125488)(6.4,1.7874511)
\psline[linecolor=black, linewidth=0.04](5.6,-0.6125488)(5.2,1.7874511)
\psline[linecolor=black, linewidth=0.04](2.4,-0.6125488)(2.8,-0.6125488)
\psline[linecolor=black, linewidth=0.04](2.8,-0.6125488)(3.2,1.7874511)
\psline[linecolor=black, linewidth=0.04](2.4,-0.6125488)(2.0,1.7874511)
\psline[linecolor=black, linewidth=0.04](3.2,-0.6125488)(3.6,-0.6125488)
\psline[linecolor=black, linewidth=0.04](3.6,-0.6125488)(4.0,1.7874511)
\psline[linecolor=black, linewidth=0.04](3.2,-0.6125488)(2.8,1.7874511)
\rput[bl](12.8,1.1874511){$K_a$}
\rput{-64.6385}(2.4013097,1.7705748){\psarc[linecolor=black, linewidth=0.04, dimen=outer](2.6,-1.0125488){0.4}{216.57303}{270.0}}
\rput{115.2305}(3.9333348,-4.5197997){\psarc[linecolor=black, linewidth=0.04, dimen=outer](3.4,-1.0125488){0.4}{216.57303}{270.0}}
\psline[linecolor=black, linewidth=0.04](2.2,-1.0125488)(3.8,-1.0125488)
\psline[linecolor=black, linewidth=0.04](2.0,-1.0125488)(1.6,-1.0125488)
\rput[bl](2.8,-1.8125489){$Q_a$}
\rput[bl](14.8,1.1874511){\Large{$X$}}
\psline[linecolor=black, linewidth=0.04, linestyle=dashed, dash=0.17638889cm 0.10583334cm, arrowsize=0.05291667cm 2.0,arrowlength=1.4,arrowinset=0.0]{->}(7.2,0.78745115)(9.4,0.78745115)
\psline[linecolor=black, linewidth=0.04, linestyle=dashed, dash=0.17638889cm 0.10583334cm, arrowsize=0.05291667cm 2.0,arrowlength=1.4,arrowinset=0.0]{->}(7.0,-2.2125487)(9.4,-2.2125487)
\rput[bl](8.2,1.1874511){$f$}
\rput[bl](4,-2.5){$\mb Q\cong \sum_{a\in A} Q_a$}
\rput[bl](8.0,-2){$h$}
\rput[bl](-1.7,0.9874512){\Large{$T\cong T\times \mb Q$}}
\rput[bl](-1.3,0){\Large{$\cong \sum_{a\in A}T\times Q_a$}}
\rput[bl](2.6,1.9874512){$T\times Q_a$}
\end{pspicture}
}

\caption{Defining $f:T\onto X$}
\label{fig:iso}
 \end{figure}

  Next, define $f:T\times \mb Q\onto X$ so that $f\uhr T\times \{q\}$ is a map $T\times \{q\}\onto K_{h(q)}$ which is either constant or an isomorphism depending on whether $K_{h(q)}$ is a singleton or a copy of $T$. We prove that $f$ is order preserving: the only non trivial thing to check is the case when $s\in T\times \{q\}$, and $t\in T\times \{q'\}$ with $q<q'$ and $a=h(q)=h(q')$. However, then $K_a$ must be a singleton by the choice of $h$ so $f(s)=f(t)$.  

 \end{proof}

The above claim proves that $T$ is strongly surjective.

\end{proof}

We should mention that whether $\diamondsuit$ suffices for the construction of a minimal Aronszajn-type is not known, and the question was already raised by Baumgartner \cite{baum}.

\section{A model without uncountable strongly surjective linear orders}\label{sec:nouncountble}

We aim to show next that, in some models of ZFC, all strongly surjective linear orders must be countable. We have seen already that strong surjectivity is closely related to minimal orders, so it is very natural to look at models of ZFC where the only uncountable minimal orders are $\omg$ and $-\omg$:  J. Moore showed that if CH and \emph{axiom (A)} holds then this is true  \cite{onlymin}. Our goal will be to prove that the same assumptions imply that there are no uncountable strongly surjective linear orders either.

First, recall that axiom (A) says that given a ladder system $(C_\alpha)_{\alpha\in \lim(\omg)}$, functions $f_\alpha:C_\alpha\to \oo$ and a Hausdorff Aronszajn tree $T$, we can find a downward closed, pruned\footnote{Recall that $S$ is pruned iff each element of $S$ has uncountably many successors in $S$.} subtree $S\subseteq T$ and $f:S \to \oo$ so that if $u\in S$ is of limit height $\alpha$ then $f(u\uhr \xi)=f_\alpha(\xi)$ for almost all $\xi\in C_\alpha$.  Such an $f$ is called a $T$-uniformization.  

It was proved in \cite{onlymin} that models of CH + (A) can be produced by starting from CH and forcing with a countable support iteration of proper posets; the individual posets introduce the uniformizations carefully so that no new reals are added in the process (not even when iterating). Therefore CH can be preserved.
  
  Our goal is to prove the following result.
  
  \begin{theorem}\label{nonexthm}
CH + (A) implies that no lexicographically ordered Aronszajn tree is strongly surjective. In particular, it is consistent that CH holds and there are no uncountable, strongly surjective linear orders.
\end{theorem}

The crux of Moore's result on minimal linear orders is \cite[Lemma 3.3]{onlymin}: CH + (A) implies that no (Hausdorff) Aronszajn tree $T$ is club-embeddable into all its downward closed, pruned subtrees. Our Lemma \ref{new33} is the surjective counterpart of \cite[Lemma 3.3]{onlymin} and essentially the content of Theorem \ref{nonexthm}.

\begin{mlemma}\label{new33}
 CH + (A) implies that there is no lexicographically ordered Aronszajn tree $T$ which can be mapped onto all of its downward closed, pruned subtrees $S$ in a lex-order preserving way. 
\end{mlemma}

The main reason we need this new lemma is the following: we do not know how to get from a lexicographic surjection to a tree embedding on a club subset; otherwise, we could have applied \cite[Lemma 3.3]{onlymin} to prove our Theorem \ref{nonexthm}. Furthermore, \cite{onlymin} deals with Hausdorff trees only while we cannot make this assumption now.

\medskip

First, let us show why Lemma \ref{new33} implies Theorem \ref{nonexthm}:

\begin{proof}[of Theorem \ref{nonexthm}]
Take any model of CH + (A) e.g. \cite[Theorem 1.9]{onlymin}. Any uncountable, strongly surjective linear order $L$ must be an Aronszajn line by Corollary \ref{cor:onlyA}  and CH. Any A-line is isomorphic to a lexicographically ordered Aronszajn tree by \cite[Theorem 4.2]{baum}. In turn, Lemma \ref{new33} implies that $L$ cannot be strongly surjective if  CH together with (A) holds.
\end{proof}

Now, we proceed to prove Lemma \ref{new33} which will be done through a sequence of claims. We will prove the following through  Claim \ref{clm0}, \ref{clm1} and \ref{clm3}: given a countable elementary submodel $M\prec  H(\aleph_2)$ and a  map $f\in M$ so that $f:T\onto S$ and $S\subseteq T$ are lexicographically ordered Aronszajn trees, one can define \emph{an unbounded branch} $b(f,M)$ of $S\cap M$ which \emph{has an upper bound in $S$}  using solely  $f^M:=f\cap M$.

Let us use the notation $\tri$ for $<_{\lex}\setm <_T$ for a lexicographically ordered tree $T$.


\begin{tclaim}\label{clm0}
 Suppose that $S\in M\prec H(\aleph_2)$ where $S$ is a lexicographically ordered Aronszajn tree. If $s\in S\cap M$ and $w\in S\setm M$ then there is $w'\in S\cap M$ so that $s\tri w'\tri w$ if $s\tri w$ and $w\tri w'\tri s$ if $w\tri  s$. 
\end{tclaim}
\begin{proof} We prove when $s\tri  w$; the other case is completely symmetric. First, note that $s\tri w$ implies  $s\tri w\uhr \vareps_0$ for some $\vareps_0\in \omg \cap M$. Now, let $$Z_r=\{t\in S:w\uhr \vareps_0\llex t\llex r\}.$$ Note that $Z_r\in M$ whenever $r\in M$. Furthermore, if there is some $\vareps\in M\cap \omg$ above $\vareps_0$ so that $Z_{w\uhr \vareps}$ is uncountable then there is some $w'\in Z_{w\uhr \vareps}$ so that $w'\tri w\uhr \vareps$; indeed, $w\uhr \vareps$ has only countably many $<_T$-predecessors. In turn, $s\tri w'\tri w$ as desired.

\begin{figure}[H]\centering
 \psscalebox{0.7} 
{
\begin{pspicture}(0,-2.537081)(8.884167,2.537081)
\psline[linecolor=black, linewidth=0.04, linestyle=dotted, dotsep=0.10583334cm](-0.12583344,1.6070809)(7.074167,1.6070809)
\psline[linecolor=black, linewidth=0.04, linestyle=dotted, dotsep=0.10583334cm](0.19416657,0.047080994)(0.57063717,0.047080994)(6.5941668,0.047080994)
\rput[bl](7.8741665,1.6070809){$M\cap \omg$}
\psbezier[linecolor=black, linewidth=0.04](1.8741666,2.407081)(1.4525295,1.4583977)(2.1408331,-1.992919)(3.4741666,-1.9929190063476563)(4.8075,-1.992919)(5.4958034,1.4583977)(5.074167,2.407081)
\rput[bl](7.4741664,0.0070809936){$\varepsilon_0\in M\cap \omg$}
\psline[linecolor=black, linewidth=0.04](3.4141665,-2.532919)(3.6141665,-1.592919)(3.2941666,-1.172919)(2.9741666,-0.412919)
\psline[linecolor=black, linewidth=0.04](3.3341665,-1.192919)(3.9141665,0.027080994)(3.4741666,1.067081)(3.9941666,2.1670809)
\psdots[linecolor=black, dotsize=0.16](4.0141664,2.187081)
\psdots[linecolor=black, dotsize=0.16](3.9141665,0.027080994)
\psdots[linecolor=black, dotsize=0.16](2.9741666,-0.412919)
\rput[bl](1.2341666,2.287081){$S$}
\rput[bl](3.4541667,2.247081){$w$}
\rput[bl](4,0.48){$w\uhr \vareps_0$}
\rput[bl](2.4341667,-0.372919){$s$}
\psline[linecolor=black, linewidth=0.04, linestyle=dashed, dash=0.17638889cm 0.10583334cm](3.0541666,0.807081)(3.7341666,0.447081)
\psdots[linecolor=black, dotsize=0.16](3.0341666,0.82708097)
\rput[bl](2.4141665,1.007081){$w'$}
\end{pspicture}
}
\caption{Finding $w'$ between $s$ and $w$}
\label{fig:w'}
\end{figure}
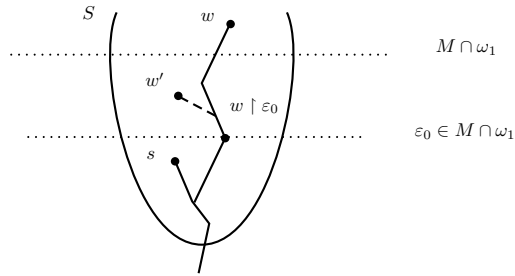

So, suppose that $Z_{w\uhr \vareps}$ is countable for all $\vareps\in M\cap \omg$ above $\vareps_0$. Hence, the set $$Z=\{r\in T:|Z_r|\leq \oo\text{ and }w\uhr \vareps_0<_T r\}$$ must be uncountable. However, for all $r,r'\in Z, $ $r\llex r'$ implies that $r\in Z_{r'}$ and so initial segments of $Z$ are countable (with respect to $<_{\lex}$). In turn, $Z$ contains a copy of $\omg$. This is a contradiction.

\end{proof}

\begin{tclaim}\label{clm1}
 Suppose $f\in M\prec H(\aleph_2)$ where $f:T\onto S$ and $S\subseteq T$ are lexicographically ordered Aronszajn trees; let $\delta=M\cap \omg$ and $h=f\uhr T_{<\delta}$. Then the following holds:
 \begin{enumerate}[(1)]
  \item  there is $u\in T_\delta$ so that  $h^{-1}(s)^\downarrow\cap u^\downarrow$ is $<_T$-bounded in $u^\downarrow$ for every $s\in S_{<\delta}$, and
  \item $f(u)\notin S_{<\delta}$ for any $u\in T_\delta$ that satisfies (1).
 \end{enumerate}

\end{tclaim}

The downward closure of a set $A\subseteq T$ will be denoted by  $A^\downarrow$,  that is: $A^\downarrow=\{t\in T: t\leq_T s$ for some $s\in A\}$.

\begin{proof} (1) Take $w\in S_\delta$ and $t\in f^{-1}(w)$; note that $t\notin T_{<\delta}$. Let  $u\in T_\delta\cap t^{\downarrow}$; we will show that $u$ works i.e. $h^{-1}(s)^\downarrow$ is $<_T$-bounded below $u$ for every $s\in S_{<\delta}$.

Fix $s\in S_{<\delta}$ and suppose that $v_n\in T_{<\delta}\cap h^{-1}(s)$ and $u_n\in v_n^\downarrow\cap u^\downarrow$ so that $\sup \{\Ht(u_n):n\in \oo\}=\delta$. We will reach a contradiction.

First, if $v_n<_{\lex}t<_{\lex} v_k$ for some $n\neq k<\oo$ then $f(t)=s\in S_{<\delta}$ which is a contradiction to $f(t)=w\in S_\delta$.

So, let us suppose that $t<_{\lex}v_n$ for all $n$. Hence $t\tri v_n$ so $f(t)=w<_{\lex} f(v_n)=s$ and so $w\tri s$ also holds. Let us pick some $w'\in S_{<\delta}$ so that $w<_{\lex} w'<_{\lex} s$; this can be done by Claim \ref{clm0}.

\begin{figure}[H]\centering
 \psscalebox{0.7} 
{
\begin{pspicture}(0,-2.9941273)(15.02002,2.9941273)
\psline[linecolor=black, linewidth=0.04, linestyle=dotted, dotsep=0.10583334cm](-0.12583344,1.3541272)(7.074167,1.3541272)
\psline[linecolor=black, linewidth=0.04](0.27416655,2.1541271)(1.4741665,-2.6458728)(5.4741664,-2.6458728)(6.6741667,2.1541271)
\psline[linecolor=black, linewidth=0.04](8.254167,2.1541271)(9.454166,-2.6458728)(13.454166,-2.6458728)(14.654166,2.1541271)
\psline[linecolor=black, linewidth=0.04, linestyle=dotted, dotsep=0.10583334cm](7.8541665,1.3541272)(15.054167,1.3541272)
\psbezier[linecolor=black, linewidth=0.04](9.854167,2.1541271)(9.432529,1.2054439)(10.120833,-2.2458727)(11.454166,-2.2458728408813475)(12.787499,-2.2458727)(13.475803,1.2054439)(13.054167,2.1541271)
\rput[bl](9.214167,2.0341272){$S$}
\psline[linecolor=black, linewidth=0.04, linestyle=dashed, dash=0.17638889cm 0.10583334cm, arrowsize=0.05291667cm 2.0,arrowlength=1.4,arrowinset=0.0]{->}(6.2741666,-0.64587283)(8.274166,-0.64587283)
\rput[bl](7.074167,-0.24587284){$f$}
\psline[linecolor=black, linewidth=0.04](2.6741667,-2.9858727)(2.0941665,-1.7058729)(2.7341666,-0.40587285)(1.9741665,0.7741272)(2.3741665,2.0341272)(2.0741665,2.7741272)
\psdots[linecolor=black, dotsize=0.16](2.0741665,2.7941272)
\psdots[linecolor=black, dotsize=0.16](2.1541665,1.3541272)
\psdots[linecolor=black, dotsize=0.16](11.254167,1.3941271)
\psline[linecolor=black, linewidth=0.04](11.254167,1.3741271)(11.534166,0.39412716)(11.194166,-0.86587286)(11.534166,-2.8658729)
\psdots[linecolor=black, dotsize=0.16](2.2341666,0.37412715)
\psdots[linecolor=black, dotsize=0.16](2.7541666,-0.38587284)
\psdots[linecolor=black, dotsize=0.16](2.2741666,-2.085873)
\psline[linecolor=black, linewidth=0.04](2.2341666,0.41412717)(2.5941665,0.9141272)
\psline[linecolor=black, linewidth=0.04](2.7541666,-0.40587285)(3.1541665,0.17412716)
\psline[linecolor=black, linewidth=0.04](2.2941666,-2.1058729)(3.3141665,-1.3658729)
\psdots[linecolor=black, dotsize=0.16](3.3141665,-1.3658729)
\psdots[linecolor=black, dotsize=0.16](3.1741667,0.17412716)
\psdots[linecolor=black, dotsize=0.16](2.5941665,0.89412713)
\rput[bl](2.1141665,-0.46587285){$u_n$}
\rput[bl](3.5741665,0.29412717){$v_n$}
\rput[bl](3.7141666,-1.4258728){$v_0$}
\rput[bl](1.5741665,-2.1658728){$u_0$}
\rput[bl](2.7741666,1.6941272){$u$}
\rput[bl](2.5541666,2.7741272){$t$}
\rput[bl](10.714167,1.7341272){$f(t)=w$}
\psline[linecolor=black, linewidth=0.04](11.334167,-1.6058729)(12.234166,-0.40587285)
\psdots[linecolor=black, dotsize=0.16](12.254167,-0.42587283)
\rput[bl](12.474167,-0.025872841){$s$}
\psline[linecolor=black, linewidth=0.04](11.354167,-0.30587283)(12.034166,0.47412717)
\psdots[linecolor=black, dotsize=0.16](12.054167,0.45412716)
\rput[bl](12.294167,0.81412715){$w'$}
\psline[linecolor=black, linewidth=0.04, linestyle=dashed, dash=0.17638889cm 0.10583334cm](2.4141665,-1.0258728)(3.0741665,-0.9058728)(3.3141665,-0.5858728)
\psdots[linecolor=black, dotsize=0.16](3.3141665,-0.62587285)
\rput[bl](3.7341666,-0.46587285){$t$}
\end{pspicture}
}
\caption{Proof of Claim \ref{clm1}}
\label{fig:clm1}
\end{figure}
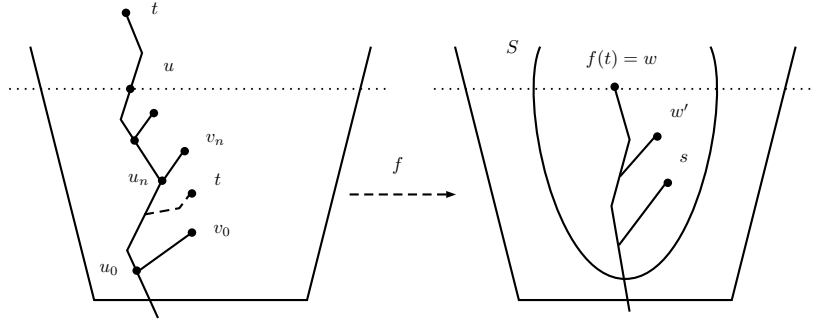


Let $t'\in h^{-1}(w')\cap T_{<\delta}$ and note that $t<_{\lex}t'$ and so $t\tri t'$ since $t\notin T_{<\delta}$. So $u_n<_{\lex} t'$ and hence $v_n<_{\lex} t'$ for some $n$. But this implies that $f(v_n)=s\leq_{\lex} f(t')=w'$ which is a contradiction.

Hence, $v_n<_{\lex}t$ for all $n<\oo$ so $s<_{\lex} w$. Again, by Claim \ref{clm0}, we can find $w'\in S_{<\delta}$ so that $s<_{\lex}w'<_{\lex} w$. Pick any $t'\in h^{-1}(w')\cap T_{<\delta}$ and note that $t'\tri t$. In turn, there must be some $n$ so that $t'<_{\lex} v_n$. However, this implies that $w'\leq_{\lex} s$, a contradiction.

(2)  This is standard using elementarity: suppose that $s=f(u)\in S_{<\delta}$ and $\vareps<\delta$. Then $H(\aleph_2)\models ``f(v)=s$ for some $v\in T$ with $u\uhr \vareps<_T v$''. So $M$ must also satisfy this sentence i.e. there is $v\in T_{<\delta}$ so that  $f(v)=s$ and $u\uhr \vareps<_T v$. However, this contradicts the assumption that the downward closure of $h^{-1}(s)$ is $<_T$-bounded below $u$.

\end{proof}

At this point, we can  define $u\in T_\delta$ with  $f(u)\in S\setm S_{<\delta}$ only using $f^M=f\cap M$; this definition can be done uniquely using a well order $\prec$ of $H(\aleph_1)$. That is, given $f\in M$ as above, we let $u=u(f,M)$ be the $\prec$-minimal element of $T=\dom(f)$ which satisfies the requirements of Claim \ref{clm1} (1) i.e. $(f\uhr M)^{-1}(s)$ is $<_T$-bounded below $u$ for every $s\in \ran(f)\cap M=\ran(f\cap M)$.

\begin{tclaim}\label{clm3} Let  $f\in M\prec H(\aleph_2)$ where $f:T\onto S$ and $S\subseteq T$ are lexicographically ordered Aronszajn trees; let $\delta=M\cap \omg$ and $h=f\uhr T_{<\delta}$.  Suppose that $u\in T_\delta$ and  $f(u)\in S\setm S_{<\delta}$. Then $$f(u)^\downarrow \cap S_{<\delta}=\bigcap\{([h(u\uhr \vareps),h(v)]_{<_{\lex}}\cap S_{<\delta})^\downarrow :\vareps<\delta,u<_{\lex}v\in T_{<\delta}\}.$$
\end{tclaim}

Here, $[h(u\uhr \vareps),h(v)]_{<_{\lex}}$ stands for all the $t\in T$ such that $h(u\uhr \vareps)\leq_{\lex} t\leq_{\lex} h(v)$.

\begin{proof}
            First, suppose that $w\in f(u)^\downarrow \cap S_{<\delta}$. Take $\vareps<\delta$ and $u<_{\lex}v\in T_{<\delta}$. Since $u\uhr \vareps <_{\lex} u$, $f(u\uhr \vareps)<_{\lex} f(u)<_{\lex} f(v)$ holds. So $H(\aleph_2)\models ``w\in z^\downarrow$ for some $z\in [f(u\uhr \vareps),f(v)]_{<_{\lex}}\cap S$''; indeed $z=f(u)$ satisfies this. So there must be some $z\in [f(u\uhr \vareps),f(v)]_{<_{\lex}}\cap S\cap M$ such that $w\in z^\downarrow$. In turn, $w\in ([h(u\uhr \vareps),h(v)]_{<_{\lex}}\cap S_{<\delta})^\downarrow$ as desired.
            
            Second, suppose that $w\in S_{<\delta}\setm f(u)^\downarrow$. We distinguish two cases: first, we consider if $w<_{\lex} f(u)$ and so $w\tri f(u)$.

            \begin{figure}[H]\centering
             \psscalebox{0.7} 
{
\begin{pspicture}(0,-2.7770655)(15.02002,2.7770655)
\psline[linecolor=black, linewidth=0.04, linestyle=dotted, dotsep=0.10583334cm](-0.12583344,1.6670656)(7.074167,1.6670656)
\psline[linecolor=black, linewidth=0.04](0.27416655,2.4670656)(1.4741665,-2.3329344)(5.4741664,-2.3329344)(6.6741667,2.4670656)
\rput[bl](9.834167,2.5270655){$S$}
\psbezier[linecolor=black, linewidth=0.04](9.354167,2.4670656)(8.800768,1.5183823)(9.620998,-1.8129344)(11.454166,-1.9329344177246093)(13.287334,-2.0529344)(14.107565,1.5183823)(13.554167,2.4670656)
\psline[linecolor=black, linewidth=0.04, linestyle=dotted, dotsep=0.10583334cm](7.8541665,1.6670656)(15.054167,1.6670656)
\psline[linecolor=black, linewidth=0.04](8.254167,2.4670656)(9.454166,-2.3329344)(13.454166,-2.3329344)(14.654166,2.4670656)
\psline[linecolor=black, linewidth=0.04, linestyle=dashed, dash=0.17638889cm 0.10583334cm, arrowsize=0.05291667cm 2.0,arrowlength=1.4,arrowinset=0.0]{->}(6.2741666,-0.3329344)(8.274166,-0.3329344)
\rput[bl](7.074167,0.06706558){$f$}
\psline[linecolor=black, linewidth=0.04](4.0941668,2.5270655)(3.3741665,0.8870656)(3.7741666,-0.9329344)(3.3341665,-2.7129345)
\psdots[linecolor=black, dotsize=0.16](4.1141667,2.5070655)
\rput[bl](2.8141665,2.5270655){$u$}
\psline[linecolor=black, linewidth=0.04](3.6141665,-1.6329345)(2.4341667,-1.1329345)(2.0741665,-0.2729344)
\psline[linecolor=black, linewidth=0.04](3.6541665,-0.29293442)(2.8341665,0.36706558)
\psdots[linecolor=black, dotsize=0.16](2.1141665,-0.3329344)
\psdots[linecolor=black, dotsize=0.16](2.8541665,0.3270656)
\rput[bl](1.6741666,-0.05293442){$x$}
\rput[bl](2.3741665,0.7070656){$y$}
\psline[linecolor=black, linewidth=0.04](11.134167,-0.3329344)(11.574166,0.8870656)
\rput[bl](9.514167,0.54706556){$h(x)=w$}
\psline[linecolor=black, linewidth=0.04](12.934167,2.3670657)(12.054167,0.46706557)(12.434167,-0.9729344)(12.054167,-2.7729344)
\psdots[linecolor=black, dotsize=0.16](12.954166,2.3470657)
\rput[bl](11.8741665,2.2670655){$f(u)$}
\psline[linecolor=black, linewidth=0.04](12.334167,-1.5529344)(11.314167,-0.8529344)(10.934167,0.3270656)
\psdots[linecolor=black, dotsize=0.16](10.934167,0.3270656)
\psdots[linecolor=black, dotsize=0.16](11.594167,0.8870656)
\rput[bl](10.614166,1.1270655){$h(y)=z$}
\end{pspicture}
}\caption{Proof of Claim \ref{clm3}}
\label{fig:clm3}
            \end{figure}
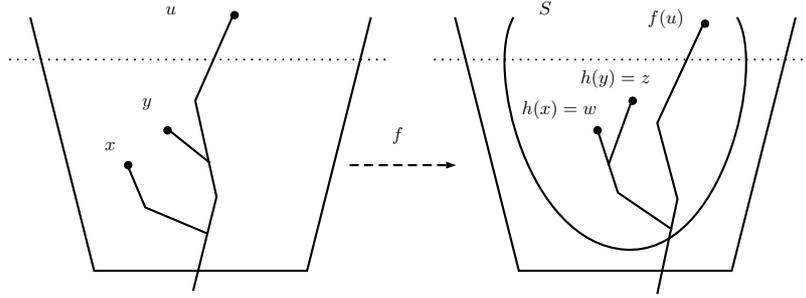

            Find some $z\in S_{<\delta}$ so that $w\tri z \tri f(u)$ holds (this can be done by Claim \ref{clm0}). If $h(x)=w$ and $h(y)=z$ for some $x,y\in T_{<\delta}$ then $x<_{\lex}y<_{\lex} u$ so $y<_{\lex} u\uhr \vareps$ for some large enough $\vareps<\delta$. The main point is that $w\notin t^\downarrow$ if $f(u\uhr \vareps)\leq_{\lex} t$. Indeed, $y<_{\lex} u\uhr \vareps$ implies that $z\leq_{\lex} f(u\uhr \vareps) \leq_{\lex} t$ so $w\tri t$ i.e. $w\notin t^\downarrow$. Hence, we found an $\vareps$ so that $w\notin ([h(u\uhr \vareps),h(v)]_{<_{\lex}}\cap S_{<\delta})^\downarrow$ for any $u<_{\lex}v\in T_{<\delta}$.
            
            The second case, when $f(u)<_{\lex} w$ is rather similar: find $z\in S_{<\delta}$ so that $f(u)\tri z\tri w$ and let $x,y\in T_{<\delta}$ so that $h(x)=w,h(y)=z$. Then $u\tri y<_{\lex} x$. Observe that $w\notin t^\downarrow$ for any $t\leq_{\lex} z=h(y)$. Hence   $w\notin ([h(u\uhr \vareps),h(y)]_{<_{\lex}}\cap S_{<\delta})^\downarrow$ for any $\vareps<\delta$.
            
           \end{proof}

 In summary, given $f\in M$ as above, we can define an unbounded branch $b(f,M)$ of $S_{<\delta}$ with an upper bound in $S$  using only  $f^M$ as follows: we set $$b(f,M)=f(u(f,M))^\downarrow\cap S_{<\delta}.$$ Claim \ref{clm1} (2) says that this is really unbounded in $S_{<\delta}$ while Claim \ref{clm3}  shows that $b(f,M)$ is definable from $f^M$.

\smallskip

Finally, we need the following

\begin{tclaim}\label{A+}
 (A) implies that any ladder system colouring $f_\alpha:C_\alpha\to \oo$  for $\alpha\in \lim(\omg)$ has a $T$-uniformization  for any Aronszajn tree $T$.
\end{tclaim}
That is, the assumption of $T$ being Hausdorff can be dropped from the definition of axiom (A). Given $\varphi:S\to \oo$ and $u\in S$ we let $\varphi[u]:\Ht(u)\to \oo$ defined by $\varphi[u](\xi)=\varphi(u\uhr \xi)$.
 \begin{proof}
  Suppose that $f_\alpha:C_\alpha\to \oo$ ($\alpha\in \lim(\omg)$) is a ladder system colouring and $T$ is an Aronszajn tree. Construct a Hausdorff tree $\tilde T$ from $T$ by inserting new, unique smallest upper bounds for bounded chains of limit length of $T$. Note that $T$ and $\tilde T$ can be uniquely recovered from one another. Let $D_\alpha=\{\xi+1:\xi\in C_\alpha\}$ and $g_\alpha:D_\alpha\to \oo$ by $g_\alpha(\xi+1)=f_\alpha(\xi)$.
  
  Now, there is a uniformization $\tilde \varphi:\tilde S\to \oo$ of $(g_\alpha)_{\alpha\in \omg}$ where $\tilde S\subseteq \tilde T$ is downward closed and pruned. Let $S=\tilde S\cap T$ and define  $\varphi:S\to \oo$ by $\varphi=\tilde \varphi\uhr S$. 
  
  If $\delta\in \lim(\omg)$ and $u\in S_\delta$ then there is $\tilde u\in \tilde S_\delta$ so that $v<_T u$ implies $v<_{\tilde T}\tilde u<_{\tilde T}u$. So $\tilde \varphi[\tilde u]\uhr D_\delta=^*g_\delta$ i.e. $\tilde \varphi(\tilde u\uhr \xi+1)=f_\delta(\xi)$ for almost all $\xi\in C_\delta$. However, note that $(\tilde u\uhr \xi+1)_{\tilde T}=(u\uhr \xi)_T$ for all $\oo\leq \xi<\delta$. So $\varphi(u\uhr \xi)=f_\delta(\xi)$ for almost all $\xi\in C_\delta$.
  
  
 \end{proof}

We are ready to prove our Lemma \ref{new33} which, given the above work, will be very similar to the original proof of   \cite[Lemma 3.3 ]{onlymin}.

\begin{proof}[of Lemma \ref{new33}] Assume that $T$ is a lexicographically ordered Aronszajn tree. We will find a subtree $S$ of $T$ so that there is no $f:T\onto S$ which preserves the lexicographic order.

Our first step is to define a map $F:H(\aleph_1)\to 2$. Fix  an arbitrary ladder system $(C_\alpha)_{\alpha\in \lim(\omg)}$. Suppose that $\mc U=(f,\varphi)$ where $f:T\onto S$ and $\varphi:S\to 2$; furthermore, let $\mc U\in M\prec H(\aleph_2)$ and $\mc U^M=(\varphi\cap M,f\cap M)$. Now, we define $F(\mc U^M)=i$ iff for all but finitely many $\xi\in C_\delta$, $\varphi[b](\xi)=i$ where $b=b(f,M)$ is the cofinal branch in $S_{<\delta}$ with an upper bound in $S$ defined above (see Claim \ref{clm3}). We set $F$ to be 0 on all other elements of $H(\aleph_1)$.

 
 Now,  \cite[Theorem 3.2]{onlymin} says that there is a $g:\omg\to 2$ so that for every $\mc U\in H(\aleph_2)$ there is a countable elementary $M\prec H(\aleph_2)$ so that $g(\omg^M)\neq F(\mc U^M)$ (where $\omg^M=\omg\cap M$). Let us define $f_\alpha:C_\alpha\to 2$ constant $g(\alpha)$. By (A) and Claim \ref{A+}, there is some (pruned, downward closed)  subtree $S$ of $ T$ and uniformization $\varphi: S\to 2$ of $(f_\alpha)_{\alpha\in \lim(\omg)}$. 
 
We claim that there is no $f:T\onto S$. Otherwise, we set  $\mc U=(f,\varphi)$ and claim that $F(\mc U^M)=g(\delta)$ for all $\mc U\in M\prec H(\aleph_2)$ and $\delta=\omg\cap M$; this would contradict the choice of $g$. So fix $M$. The fact that the branch $b=b(f,M)$ of $S_{<\delta}$ has an upper bound in $S$ implies that $\varphi[b]\uhr C_\delta=^*f_\delta=g(\delta)$ since $\varphi$ was a uniformization. In turn, we defined $F(\mc U^M)=g(\delta)$. 
\end{proof}

\section{Open problems}

Regarding suborders of $\mb R$ and the results of Section \ref{sec:real}, the following is very natural:

\begin{con}
 $\textmd{MA}_{\aleph_1}$ does not imply the existence of uncountable, strongly surjective real suborders.
\end{con}

\begin{prob}
 Does every uncountable (real) strongly surjective order contain a minimal suborder?
\end{prob}

\begin{prob}
Is every short, homogeneous and minimal (real) linear order strongly surjective?
\end{prob}

Section \ref{sec:allAronszajn} motivated the next question:

\begin{prob}
 Are there any uncountable strongly surjective linear orders in the Cohen-or other canonical models (Sacks, Miller, etc.)?
\end{prob}



      \begin{prob}
Could there be a strongly surjective $L$ of size $\mf c>\aleph_1$?
\end{prob}

In particular, we ask for a short linear order of size $\mf c>\aleph_1$ without real suborders of size $\mf c$ (by Corollary \ref{cor:noc}). Such linear orders can be constructed, say with $\mf c=\aleph_{\omg}$; we only sketch the proof. 

\begin{prop}
 Suppose that $V\models CH$. Then, in the model $V^{\mb C_{\aleph_{\omg}}}$, $\mf c=\aleph_{\omg}$ and there is a short linear order of size $\mf c$ without real suborders of size $\mf c$.
\end{prop}
\begin{proof}
$L$ is defined as the lexicographic order on a tree $T$ with the following properties: $T$ has height $\omg$, size $\mf c$ and levels of size $<\mf c$. Furthermore, each level $T_\alpha$ is separable and $T$ has no $\aleph_1$ chains; these properties ensure that the linear order is short and has no separable suborders of size $\mf c$.

Now, to construct $T$ using a generic  $G\subseteq {\mb C_{\aleph_{\omg}}}$, we inductively define $T_\alpha$ for $\alpha<\omg$. Given $T_{<\alpha}$ we select a cofinal copy $S_\alpha$ of $2^{<\oo}$ from $T_{<\alpha}$ and use $G\uhr [\omega_\alpha,\oo_\alpha+\oo_\alpha)$ to find $\aleph_\alpha$-many generic branches through $S$. These branches give $T_\alpha$. 

The only non trivial property to check is that there are no $\aleph_1$ chains in $T$. However, note that there is a club $C\subseteq \omg$ so that if $\alpha\in C$ and $b\in V[G\uhr \oo_\alpha]$ is a cofinal branch through $T_{<\alpha}$ then $b$ has no upper bound in $T_\alpha$; this follows from genericity. In turn, there could be no $\aleph_1$-chains.
\end{proof}

On the other hand, the following  holds.

\begin{prop}
 Suppose  that $\aleph_1<\mf c$ and $$\cf([\lambda]^\oo,\subseteq)<\cf(\mf c)$$ for all $\lambda<\mf c$. Then any short linear order $L$ of size $\mf c$ contains a real suborder of size $\mf c$.
\end{prop}

The assumptions of the Proposition are satisfied if $\aleph_1<\mf c<\aleph_\oo$ since $\cf([\aleph_n]^\oo,\subseteq)=\aleph_n$.

\begin{proof}First, note that $\aleph_1<\cf(\mf c)$. Let $T$ be an everywhere 2-branching partition tree for $L$. Then $T$ has height $\leq \omg$ so there is a minimal $\alpha<\omg$ such that $T_\alpha$ has size $\mf c$. So $\lambda=|T_{<\alpha}|<\mf c$ and hence there is cofinal family in $[T_{<\alpha}]^\oo$ of size $<\cf(\mf c)$. Also, any element of $T_\alpha$ is given by a branch of a countable subset of  $T_{<\alpha}$. In particular, there is a single countable $S\in [T_{<\alpha}]^\oo$ so that there are $\mf c$ many branches through $S$ with upper bounds in $T_\alpha$. This gives a real suborder of size $\mf c$.
\end{proof}


\medskip

Also, the following question is open although it could be as hard as proving the consistency of Baumgartner's axiom for $\aleph_3$-dense sets of reals:

\begin{prob}
Construct strongly surjective orders of size $>\aleph_2$. Can a strongly surjective linear order have size $2^{\aleph_1}$?
\end{prob}

\begin{center}
 $\star$
\end{center}

The construction of  Section \ref{sec:diamond} raises the following question which was already asked by Baumgartner \cite{baum}:

\begin{prob}
 Does $\diamondsuit$ or the existence of a Suslin tree suffice to show that there is a minimal Aronszajn-type?
\end{prob}

\begin{center}
 $\star$
\end{center}
%
%

  J. Moore showed that, under PFA, uncountable linear orders have a 5 element basis: any uncountable linear order must embed $\omg,-\omg$, or an uncountable real order type, or a fixed Countryman line $C$ or its reverse $-C$. Also, already under $\textmd{MA}_{\aleph_1}$, there are minimal Countryman lines \cite[Theorem 2.2.5]{walks}. 
  
  So the next questions are rather natural:
  
  \begin{prob}
   Does $\textmd{MA}_{\aleph_1}$ imply that there is an uncountable strongly surjective linear order?
  \end{prob}

    \begin{prob} Can a strongly surjective linear order $L$ be a Countryman line i.e. is it possible that $L^2$ is the union of countably many chains?
     
    \end{prob}
    
\begin{center}
 $\star$
\end{center}

  Now, one can easily refine the notion of being strongly surjective by requiring the existence of maps for only a restricted class of suborders. Let us say that $L$ is \emph{surjective for the class $\mc K$} iff $L\onto K$ for any $K\in \mc K$ such that $K\subseteq L$. So $L$ is strongly surjective iff it is surjective for the class of all linear orders. We say that $L$ is $\kappa$-surjective iff $|L|\geq \kappa$ and $L$ is surjective for the class of all linear orders of size $\kappa$. 

Note that every cardinal $\kappa$ with its usual well order is  $\kappa$-surjective but not $\lambda$-surjective for $\lambda<\cf(\kappa)$. In particular, $\omg$ is $\aleph_1$-surjective but not strongly surjective.

\begin{prob}
  Is there a short, $\aleph_1$-surjective linear order $L$  which is not strongly surjective i.e. $L\not \onto \mb Q$?
\end{prob}

It is easy to see that a set of reals of size $\aleph_1$ is strongly surjective iff it is $\aleph_1$-surjective.

\begin{prob}
 Suppose that $L\subseteq \mb R$ is $\aleph_2$-surjective of size $\aleph_2$. Is $L$ strongly surjective?
\end{prob}

\begin{center}
 $\star$
\end{center}
    
    The following problems concern the question if strong surjectivity reflects: 
  
 \begin{prob} Suppose that $L$ is strongly surjective and $x\in L$. Is $L\setm \{x\}$ strongly surjective?
  \end{prob}

  \begin{prob}
Suppose that $\oo\leq \lambda<\kappa$ and $L$ is a strongly surjective linear order of size $\kappa$. Is there a strongly surjective suborder of $L$ of size $\lambda$?
  \end{prob}

  Yes, for $\lambda=\oo$ trivially (either $\oo$ or $-\oo$ embeds into $L$, and also $\mb Q$ embeds into any uncountable, short linear order by an old result of Hausdorff).

\begin{center}
 $\star$
\end{center}

Finally, about mixing the order types of strongly surjective orders, we ask:

\begin{prob}
 Is it consistent that there are real and Aronszajn strongly surjective linear orders at the same time?
\end{prob}

 \begin{prob}
   Is it consistent that there are uncountable, strongly surjective linear orders but each such order is separable?
  \end{prob}




\begin{thebibliography}{2}

\bibitem{abMA} Abraham, U., Shelah, A. "Martin's axiom does not imply that every two $\aleph _{1}$-dense sets of reals are isomorphic" Israel J Math 38 (1981) 161-176

\bibitem{ARSh} Abraham, Uri, Matatyahu Rubin, and Saharon Shelah. "On the consistency of some partition theorems for continuous colorings, and the structure of $\aleph_1$-dense real order types." Annals of Pure and Applied Logic 29.2 (1985): 123-206.


\bibitem{baumiso} Baumgartner, James E. "All $\aleph_1$-dense sets of reals can be isomorphic." Fundamenta Mathematicae 79.2 (1973): 101-106.

\bibitem{baum} Baumgartner, James E. "Order types of real numbers and other uncountable orderings." Ordered sets. Springer Netherlands, 1982. 239-277.


\bibitem{devlin} Devlin, Keith J.; Shelah, Saharon A weak version of $\diamondsuit$ which follows from $2^{\aleph_0}<2^{\aleph_1}$. Israel J. Math. 29 (1978), no. 2-3, 239–247.


\bibitem{osvaldo} O. Guzman, M. Hru\v s\'ak "Parametrized $\diamondsuit$-principles and canonical models." Slides from Retrospective workshop on Forcing and its applications, Fields Institute, 2015.

\bibitem{hns} Hajnal, A., Zs Nagy, and L. Soukup. "On the number of certain subgraphs of graphs without large cliques and independent subsets." A Tribute to Paul Erdos (1990): 223.

\bibitem{jech} Jech, T. "Set theory. The third millennium edition, revised and expanded." Springer Monographs in Mathematics. Springer-Verlag, Berlin, 2003. xiv+769 pp. ISBN: 3-540-44085-2 

\bibitem{onlymin} Moore, Justin Tatch. "$\omg$ and $-\omg$ may be the only minimal uncountable linear orders." Michigan Math. J 55.2 (2007): 437-457.

\bibitem{5basis} Moore, Justin Tatch. "A five element basis for the uncountable linear orders." Annals of Mathematics (2006): 669-688.


\bibitem{DHM} Moore, Justin, Michael Hrušák, and Mirna Džamonja. "Parametrized $\diamondsuit$ principles." Transactions of the American Mathematical Society 356.6 (2004): 2281-2306.


\bibitem{raphael0} Camerlo, Riccardo; Carroy, Raphael;  Marcone, Alberto. "Epimorphisms between linear orders." Order 32 (2015), 387-400. Erratum, Order 33 (2016), 187.

\bibitem{raphael} Camerlo, Riccardo; Carroy, Raphael;  Marcone, Alberto; "Linear orders: when embeddability and epimorphism agree." preprint, arXiv:1701.02020


\bibitem{stevotrees} Todorcevic, Stevo. "Trees and linearly ordered sets." Handbook of set-theoretic topology (1984): 235-293.

\bibitem{walks} Todorcevic, Stevo. "Walks on ordinals and their characteristics." Vol. 263. Springer Science \& Business Media, 2007.


\end{thebibliography}
\end{document}